\documentclass[a4paper]{article}
\usepackage[T1]{fontenc}
\usepackage{lmodern}
\usepackage{geometry}

\usepackage{calc}
\usepackage{lastpage}
\usepackage{fancyhdr}
\usepackage{lmodern,ifthen}
\usepackage[svgnames]{xcolor}
\usepackage{url}
\usepackage{makeidx}
\usepackage{ifthen}
\usepackage[utf8]{inputenc}
\usepackage[english]{babel}
\usepackage[english]{layout}
\usepackage{amsfonts}
\usepackage{amsmath,bbold,amsthm,enumerate, amssymb}
\usepackage{graphicx, url}
\usepackage{wallpaper, enumitem}
\usepackage{natbib}
\usepackage[nottoc, notlof, notlot]{tocbibind}
\usepackage{caption}
\usepackage{algorithm}
\usepackage[noend]{algpseudocode}
\usepackage{hyperref}
\usepackage{indentfirst}
\usepackage{mathtools, stmaryrd}
\usepackage{xparse} 
\usepackage{authblk}
\usepackage{subfig}

\makeatletter
\newcommand{\mathleft}{\@fleqntrue\@mathmargin0pt}
\newcommand{\mathcenter}{\@fleqnfalse}
\makeatother

\def\AIC{\textsc{aic}}
\def\BIC{\textsc{bic}}

\DeclarePairedDelimiterX{\Iintv}[1]{\llbracket}{\rrbracket}{\iintvargs{#1}}
\NewDocumentCommand{\iintvargs}{>{\SplitArgument{1}{,}}m}
{\iintvargsaux#1} %
\NewDocumentCommand{\iintvargsaux}{mm} {#1\mkern1.5mu..\mkern1.5mu#2}
\newcommand{\overbar}[1]{\mkern 1.5mu\overline{\mkern-1.5mu#1\mkern-1.5mu}\mkern 1.5mu}

\DeclareMathOperator*{\argmin}{argmin}   
\DeclareMathOperator*{\argmax}{argmax} 
\newcommand{\Proba}{pr}
\newcommand{\Kl}{(k_l)_l\in \Iintv{1,K}^L}
\graphicspath{{./}}

\hypersetup{
    colorlinks=true,
    linkcolor=red,
    filecolor=magenta,      
    urlcolor=cyan,
    citecolor = blue
}

\newtheorem{theorem}{Theorem}[section]

\newtheorem*{assumption*}{Assumption}
\newtheorem{assumption}{Assumption}[section]
\newtheorem{proposition}{Proposition}[section]

\newtheorem{lemma}{Lemma}[section]

\title{Clustering and Model Selection via Penalized Likelihood for
  Different-sized Categorical Data Vectors}
\author[1]{Esther Derman\thanks{esther.boccara@polytechnique.edu}}
\author[1]{Erwan Le Pennec\thanks{erwan.le-pennec@polytechnique.edu}}
\date{}

\affil[1]{CMAP, Ecole polytechnique, CNRS, Universit\'e Paris-Saclay, 91128, Palaiseau, France.}

\makeatletter
\def\BState{\State\hskip-\ALG@thistlm}
\makeatother

\begin{document}
\maketitle

\begin{abstract}
In this study, we consider unsupervised clustering of categorical
vectors that can be of different size using mixture.
We use likelihood
maximization to estimate the parameters of the underlying mixture
model and a penalization technique to select the number of mixture
components.
 Regardless of the
true distribution that generated the data, we show that an explicit penalty,
known up to a multiplicative constant, leads to
 a non-asymptotic oracle inequality with the
Kullback-Leibler divergence on the two sides of the inequality. This theoretical result is illustrated by a document
clustering application. To this aim a novel robust
expectation-maximization algorithm is proposed 
to estimate the mixture parameters that best represent the
different topics. 
Slope heuristics are used to calibrate the penalty and to select a
number of clusters.
\bigbreak
 
\end{abstract}

\section{Introduction}
\subsection{Motivation}
This study explores unsupervised clustering and model selection for multidimensional categorical data. Our basic block will be observations $X_l$ of $n_l$ independent and identically distributed instances of a categorical variable. We will assume we observe $L$ independent such random vectors $(X_1,\ldots,X_L)$, that they share the same categories but not necessarily the same distribution and the same size. Such an observation is equivalent to a vector of $L$ independent multinomial random variables, i.e. $L$ barplots, with the same number of modalities but not necessarily the same distribution and the same size.  Our objective is to cluster these observations, or the corresponding barplots, according to their similarities. 

Clustering of categorical data is used in many fields such as social sciences, health, genetics, text analysis, etc. Several clustering techniques exist in the literature such as K-means or hierarchical clustering to name a few, we have chosen to use a mixture model to cluster the data. The first challenge encountered in our framework is that our observations $X_l$s are not of the same size. In other words, all the barplots do not necessarily represent the same number of instances. To our knowledge, this particularity regarding clustering has not yet been introduced in the literature. \citet{GasRouVer2016} studied mixtures of multidimensional vectors in the case of independent and continuous marginal densities where they assume every observation has the same number of coordinates. In our framework, the $X_l$s may be of different dimensions besides being categorical, although a generalization to continuous variables should lead to similar properties. The most comparable work into this direction is the one of \citet{MatRebVil2015} who model interaction events between individuals assuming they are clustered according to their pairwise interactions. They approximate interaction intensities by piecewise constant functions and penalize the resulting likelihood to do model selection. 

A second challenge when dealing with our mixture model is the classical problem of choosing a proper number of clusters. An expectation-maximization algorithm is commonly used to estimate the best parameters in a given model \citep{McLPee2000}, but there is no direct method to do model selection. Several techniques have been studied to select the number of clusters efficiently. \citet{SilCarFig2014} propose a data-driven criterion for choosing the optimal number of clusters for categorical data. Based on a minimum message length criterion, they incorporate a model selection process inside the iterations of an expectation-maximization algorithm. Thanks to this method, they optimize time-consuming computation and avoid several runs of the same algorithm. On the same vein, \citet{RigCapYvo2006} estimate parameters of a multinomial mixture model for document clustering in a Bayesian framework and consider the model selection problem of choosing a dictionary that best discriminates the documents. Based on a text corpus, they alleviate the problem of high variability of the estimates in high dimension by deleting the most rare words from the dictionary. Using a categorical dataset of low back pain symptoms, \citet{FopSmaMur2017} address a Bayes factor criterion to select the most discriminative symptoms and cluster patients according to their symptoms' similarities. 

We consider a penalized maximum likelihood approach where the penalty addresses the problem of simultaneously selecting the best multinomial mixture within a collection and the best group assignment of the vectors underlying the selected mixture. A classical trade-off between bias and variance naturally appears, which depends on this twofold complexity. Our model selection criterion takes these two steps into account. We show that it estimates mixture parameters efficiently and that vectors are robustly assigned to a proper cluster whatever the size of the data.

\subsection{Results}

The main result of this paper is a non-asymptotic oracle inequality that gives a sufficient condition on the penalty such that our estimator performs almost as well as the best one. Non-asymptotic oracle inequalities for categorical observations have already been introduced in \citet{TouGas2009} and \citet{BonTou2013} in the context of genomics, where the authors cluster a population into subpopulations according to their alleles' categories. The observations studied in their analysis are considered as independent and identically distributed vectors that are equally informative and of the same length as in our setting, they are conditionally independent and the information our data give depends on the length of the observed variables. Moreover, we aditionaly take into account cluster assignment of the observations when evaluating the risk of our estimation.

Our theoretical work is built on Massart's methodology \citep{Mas2007} to compute the penalty function. Technical proofs are also inspired by the technical report of \citet{CohLep2011} (see also \citet{CohLep2012}) in which the authors use a penalized model selection technique with a maximum likelihood approach for conditional density estimation in a random design setting. Assuming we know which of the $L$ density laws generated each observation, we place ourselves in a fixed design setting with deterministic covariates, although our result can easily be generalized to random covariates cases. Under mild assumptions on model structure, the non-asymptotic oracle inequality we address keeps the form of a Kullback-Leibler risk on the left side of the inequality without passing by a weaker one as it was the case in previous works (see \cite{Mas2007}, \cite{MauMic2008Pen}, \cite{CohLep2012}, \cite{BonTou2013} and \cite{MonLeP2014}). 

This theoretical study is applied to document clustering of NIPS conferences from 1987 to 2015. In order to estimate the parameters, we use a robust expectation-maximization algorithm that avoids classical numerical issues of local maxima and convergence to space boundaries. Moreover, since our theoretical penalty is known up to a multiplicative constant, we calibrate an optimal one by using slope heuristics and select the best model accordingly. The resulting visualization of different topics result to our intuitive a priori on their characteristics and effectively describes an evolution of machine learning overtime.

\section{Model Selection Criterion}
\label{sectionThm}

\subsection{Framework and notations}
In our model, we observe a family of $L$ independent random vectors $(X_1,\ldots,X_L)$ where each $X_l$ represents $n_l$ independent and identically distributed instances $X_l^i$ of a random variable having $s_l$ as a true categorical density distribution with respect to the counting measure over B modalities, i.e.
\begin{align*}
  \Proba\left((X_1,\ldots,X_L) = (x_1,\ldots, x_L)\right) 
                                                           &= \prod_{l=1}^L \Proba\left((X_l^1,\ldots,X_l^{n_l})=(x_l^1,\ldots,x_l^{n_l})\right)\\
                                                          & = \prod_{l=1}^L \prod_{i=1}^{n_l} s_l(x_l^i).
\end{align*}
This corresponds to the observation of $n = \sum_{l=1}^L n_l$ independent but only groupwise identical random variables. 
Assume, for a moment, that those $L$ vectors can further be regrouped in $K$ groups sharing the same density distribution. A natural model for such a situation is a mixture model in which the latent group comes from a $K$-dimensional multinomial variable $Z$ that takes values in the set $\{1,\ldots, K\}$ of the different group labels and each group is characterized by its own density distribution $f_k$.  Let $\pi_k = \Proba(Z=k)$. The distribution of $Z$ is determined by the vector $\pi = (\pi_1,\ldots, \pi_K)$ which  belongs to the $(K-1)$-dimensional simplex $\mathbb{S}_{K-1}$ and the density distribution becomes
\begin{align*}
  \label{mixtureDistrib}
   \Proba\left((X_1,\ldots,X_L) = (x_1,\ldots, x_L)\right) &=
\prod_{l=1}^{L} \Proba(X_l=x_l)   = \prod_{l=1}^{L} \left(\sum_{k=1}^{K}   \pi_k  \Proba (X_l = x_l\mid Z_l = k)\right) \\
& = \prod_{l=1}^{L} \left\{ \sum_{k = 1}^{K} \pi_k \left( \prod_{i =
                                                                                                                1}^{n_l}f_k(x_l^i)\right) \right\}.
\end{align*}
It is then usual to assign to each vector $X_l$ its group tanks to a
maximum a posteriori principle:
\begin{equation*}
  k_l = \argmax_{k\in \{1,\ldots, K\}} \Proba(Z=k | X_l = x_l) = \argmax \left\{\pi_k \left( \prod_{i =
                                                                                                                1}^{n_l}f_k(x_l^i) \right)\right\}
\end{equation*}
We will never assume that such a model holds but use the
corresponding density functions as approximations of the true one and
infer a group for each vector $X_l$ with the same maximum a posteriori principle. For
any $K$, such a mixture density distribution is entirely specified by
$\theta=(\pi,f)=\left((\pi_1,\ldots,\pi_K),(f_1,\ldots,f_K)\right)$ in
$\Theta_{K} = \mathbb{S}_{K-1} \times \mathcal{F}_{K}$ with
$\mathcal{F}_{K}$ referring to a set of $K$ categorical functions that
will be specified later on. With a slight abuse of notation, we will
denote by $P_{(K, \theta)}(x)$ the likelihood of an observation
\begin{equation*}
\label{likelihood1Obs}
\begin{split}
P_{(K, \theta)}(x) &= \prod_{l=1}^L P_{(K, \theta)}(x_l) =
\prod_{l=1}^L  \left\{ \sum_{k = 1}^{K} \pi_k \left( \prod_{i = 1}^{n_l}f_k(x_l^i)\right)\right\}.
\end{split}
\end{equation*}
For a fixed $K$, the distribution density is chosen in the model $\mathcal{S}_K =
\{P_{(K, \theta)}: \theta \in \Theta_{K}\}$, while $K$ has to be
chosen among the set of positive integers smaller than $L$. Our task is
thus to estimate a good couple $(K,\theta)$.
 
By construction, $f_k$ is such that the class of multinomial
mixtures is identifiable as long as $B \geq 2K-1$, meaning that it is
fully determined by its parameters up to label switching, as  shown in \cite{ElmWan2003}.
\subsection{Maximum likelihood estimation}

We use a classical maximum likelihood estimator to obtain the mixture
parameters and use those parameters to assign groups.

More precisely, we define as $\gamma_n(\pi, f)$ the empirical contrast:
\begin{equation*}
\begin{split}
\label{ll}
\gamma_n(\pi, f) = \sum_{l = 1}^{L}-\log  \left(P_{(K, \theta)}(x_l) \right).
\end{split}
\end{equation*}
The estimated parameter denoted by $(\widehat{\pi}, \widehat{f})$ maximizes the log-likelihood of the observations, which is equivalent to minimizing the empirical contrast:
\begin{equation*}
\label{loglike}
(\widehat{\pi}, \widehat{f})= \argmin_{(\pi ,f )\in \Theta_K} \gamma_n(\pi, f).
\end{equation*}
To avoid existence issue, we will work with an almost minimizer and define an $\eta$ -log-likelihood minimizer as any $(\widehat{\pi}, \widehat{f})$ that satisfies:
\begin{equation}
\label{etaMLE}
\gamma_n(\widehat{\pi}, \widehat{f})     \leq   \inf_{(\pi ,f )\in \Theta_K}   \gamma_n(\pi, f)   + \eta,\quad\eta>0.
\end{equation}

Since maximum likelihood estimation assigns a zero probability to unobserved categories, the likelihood can be infinite in some cases. This drawback is classical in discrete settings. To avoid this issue, we make the following assumption:

\begin{assumption}[Model Structure]
\label{hypSurModele}
Any candidate density $f$ satisfies $f\geq \epsilon_n = e^{-\tau_n}$.

\end{assumption}

The value $\epsilon_n$ can typically be taken as $1/n$. Thus, $-\log(f) \leq \tau_n$, where $\tau_n = \log(n)$. This assumption is legitimate in a context of document clustering because categories correspond to words that appear in at least one text. 

We can now define a set  $\mathcal{F}_{K}$ of $K$ density functions
on $B$ categories so that any candidate satisfies Assumption
\ref{hypSurModele}. This set $\mathcal{F}_{K}$ is a simple product of
multinomial densities satisfying the assumption:
\begin{equation*}
\begin{split}
\mathcal{F}_K = \left\{(f(1),\ldots,f(B)) ; \sum_{b = 1}^{B} f(b) = 1 , e^{-\tau_n} \leq f(b) \leq 1 \right\}^K = \mathcal{F}^K.
\end{split}
\end{equation*}
Once parameters are estimated, an observation $x_l$ is assigned to the cluster it most likely belongs to. We use the maximum a posteriori method:
\begin{equation*}
\widehat{k_l} = \argmax_{k\in \Iintv{1,K}}  \left\{\widehat{\pi}_k \left(\prod_{i = 1}^{n_l} \widehat{f}_k  (x_l^i)\right)  \right\}.
\end{equation*}
Finally, we note $\widehat{f}_{\widehat {k}_l}$ the estimated distribution of observation $x_l$.


\subsection{Single model risk bound}

In this section, we prove that for a fixed $K$ such a scheme allows controlling the error between the true density function by a multiple of
the best possible mixture. We will measure the error with the 
Kullback-Leibler loss. Recall that Kullback-Leibler divergence is a natural quality measure to be considered in maximum likelihood approach. It is defined by
\begin{equation*}
\mathbf{KL}(s, t) = \int \log \left( \frac{s}{t}\right)sd\lambda,
\end{equation*}
where $\lambda$ is a known positive measure such that $sd\lambda$ is
absolutely continuous with respect to $td\lambda$.

Our first result is an oracle inequality that upper bound the
expectation of Kullback-Leibler divergence between our estimate and
the truth by the best possible one in the model $\mathcal{S}_K$ up to
a multiplicative factor and a \emph{variance} term that measures the
model complexity.  \cite{CohLep2012} show that under some general
assumption on this bracketing entropy, model complexity is
proportional to the model dimension. In our case, as we have an
assigment step, the dimensions of the considered models correspond to
the number of mixture parameters plus the cluster assignment cost.
The dimension of a mixture model with $K$
categorical functions of $\mathcal{F}_{K}$ is $KB-1$. For sake of
simplicity, we will rather use a \emph{dimension} $D_K = KB$.

\begin{theorem}
\label{thmSingleModelParticular}
Consider the observed vectors $(x_1,\ldots,x_L)$ as above and denote by $n$ the overall number of observations. Consider also one model $\mathcal{S}_{K}$ as defined above and assume it satisfies Assumption \ref{hypSurModele}. Let $(\widehat{\pi}, \widehat{f})$ be an $\eta$-likelihood minimizer in $\mathcal{S}_{K}$ as defined in equation \refeq{etaMLE}.
Denote the corresponding cluster assignment for each vector $x_l = (x_l^1,\ldots,x_l^{n_l})$:
\begin{equation*}
\widehat{k_l} = \argmax_{k \in \Iintv{1,K}} \left\{ \widehat{\pi_k} \left(\prod_{i = 1}^{n_l} \widehat{f_k}(x_l^i) \right) \right\}.
\end{equation*}
Then, for any $C_1>1$, there exist two constants $\lambda_0$ and $C_2$ depending only on $C_1$ such that the estimate $(\widehat{\pi}, \widehat{f})$ satisfies
\begin{equation*}
\begin{split}
E\left[\sum_{l = 1}^{L} \frac{n_l}{n} \mathbf{KL}(s_l, \widehat{f}_{\widehat{k_l}}) \right]     &\leq  C_1  \left\{ \inf_{\substack{f \in \mathcal{F}_K\\(k_l)_l \in \Iintv{1,K}^L}}  \left(\sum_{l=1}^{L} \frac{n_l}{n} \mathbf{KL}(s_l, f_{k_l})\right)       + \lambda_0\frac{  L\log K + \mu_nD_{K}}{n} \right\} \\
&\quad+ \frac{C_2}{n} + \frac{ \eta}{n},
\end{split}
\end{equation*}
where $\mu_n = 2(\log(2\tau_n)^{1/2} + \pi^{1/2})^2 + 1 + \log(n)$.
\end{theorem}

In previous works (see \cite{Mas2007}, \cite{CohLep2012}, \cite{BonTou2013} and  \cite{MauMic2008Pen}), a lower divergence appears on the left side of the inequality, such as Hellinger distance or Jensen-Kullback-Leibler divergence. The main contribution of our result is that the empirical risk of the estimated density is measured according to the same averaged Kullback-Leibler divergence as the risk model appearing on the right side of the inequality. The upper bound addressed here is sharper and explicits more clearly the bias underlying density estimation because risk functions are the same on the left and right side of the inequality. This is made possible thanks to Assumption \ref{hypSurModele}, which enables to bound the moments of log-ratios of density distributions. A concentration inequality on the empirical Kullback-Leibler divergence can then be deduced. Without Assumption \ref{hypSurModele}, density distributions must be weighted by the true ones in order to bound a log-ratio which leads to a weaker divergence than \textbf{KL} (\cite{Mas2007} and \cite{CohLep2012}).

Moreover, this oracle inequality takes the two steps of the estimation into account, namely mixture parameters and clustering assignment of the observations. This explains the term in $L\log(K)$ appearing on the left side of the inequality. It measures the cost for assigning each observation to one cluster, as $D_K$ measures the complexity of mixture models. The overall \textit{variance} $\lambda_0\left(  L\log K + \mu_nD_{K}\right) $ clearly balances the bias term appearing in the inequality as the minimum of the divergence over the parameter set.  

Theorem~\ref{thmSingleModelParticular} turns out to be crucial in our
analysis because the model complexity that appears explicitly is the
one that we will use to define a suitable penalized criterion.

\subsection{Model selection theorem}

We focus now on the choice of $K$ and prove that this can be done with
a simple penalization of the likelihood.
Indeed, for any $K$, the oracle inequality of
Theorem~\ref{thmSingleModelParticular} holds as soon as
Assumption~\ref{hypSurModele} holds.  One of the models minimizes the
right hand side of the inequality but there is no way of knowing which
one without also knowing the true densities $s_1,\ldots,s_L$.

We propose thus a data-driven strategy to select $K$ that performs
almost as well as we had known the best $K$. We use a penalized criterion
\begin{equation*}
\mathbf{crit}(K) = \min_{(\pi, f) \in  \Theta_K}\gamma_n(\pi, f)  + \mathbf{pen}(K),
\end{equation*}
where $\mathbf{pen}: K \rightarrow \mathbb{R}^+$ denotes the penalty function. 
Our analysis suggests the use of a penalty of the form $\mathbf{pen}(K) =
\lambda_0( D_K + L\log(K))$, where $D_K$ plays the role of mixture
complexity and $L\log(K)$ corresponds to cluster assignment cost,
similar to the quantity . This relates to classical penalties such as $\AIC = D_K/n$ and $\BIC = D_K\log(n)/(2n)$ that are proportional to model dimension, but such penalties require large sample sizes in order to be consistent. The penalty we propose is consistent with non-asymptotic sample size and performs well with respect to a theoretical loss function $\mathbf{KL}$. 

The main result of this study is the following theorem that compares
the risk of the selected model with the risk of the unknown best model
and shows that for a suitable choice of the constant in the penalty function, the estimated model still performs almost as well as the best one.

\begin{theorem}
\label{thmModelSelectionParticular}
Consider the observed vectors $(x_1,\ldots,x_L)$ and denote by $n$ the
overall number of observations. Consider also the collection
$(\mathcal{S}_K)_{K \in\Iintv{1,L}}$ of models satifying Assumption~\ref{hypSurModele} defined above.

Let $\mathbf{pen}$ be a non-negative penalty function and  $\hat{K}$  any
$\eta$-minimizer of
\begin{equation*}
\mathbf{crit}(K) = \min_{(\pi, f) \in  \Theta_K}\gamma_n(\pi, f)  +
\mathbf{pen}(K).
\end{equation*}
Let $(\widehat{\pi}, \widehat{f})$ be the corresponding
$\eta$-likelihood minimizers in $\mathcal{S}_{\widehat{K}}$ and define
the resulting cluster assignment for each vector $x_l = (x_l^1,\ldots,x_l^{n_l})$:
\begin{equation*}
\widehat{k_l} = \argmax_{k \in\Iintv{1,\widehat{K}}} \left\{ \widehat{\pi_k} \left(\prod_{i = 1}^{n_l} \widehat{f_k}(x_l^i) \right) \right\}.
\end{equation*}
Then, for any constant $C_1 >1$, there exist two constants $\lambda'_0$ and $C_2$ depending only on $C_1$ such that if the penalty function is defined as:
\begin{alignat*}{2}
  \mathbf{pen}: 
  K &\longmapsto& \lambda'_0 \left(\mu_nD_{K} +L\log( K) + K\log(2) \right)
\end{alignat*}
with $\mu_n = 2(\log(2\tau_n)^{1/2} + \pi^{1/2})^2 + 1 + \log(n)$, then, whatever the underlying true densities $s_1,\ldots,s_L$, 
\begin{equation*}
\begin{split}
E\left(\sum_{l =1}^{L}  \frac{n_l}{n} \mathbf{KL}(s_l, \widehat{f}_{\widehat{k_l}} ) \right)
&\leq C_1    \inf_{K} \left\{  \inf_{\substack{f \in \mathcal{F}_K\\(k_l)_l \in \Iintv{1,K}^L}}  \left(\sum_{l = 1}^{L}\frac{n_l}{n}  \mathbf{KL}(s_l, f_{k_l})  \right) +  \frac{\mathbf{pen}(K)}{n}
 \right\}  \\
 &\quad+ \frac{C_2}{n} +\frac{\eta}{n} .
\end{split}
\end{equation*}
\normalsize
\end{theorem}

The inequality looks similar to the one of
Theorem~\ref{thmSingleModelParticular} and indeed the two differences
are the infimum over $K$, which was not present in the first theorem and the penalty function, which is equal to
the \emph{variance} term of the first theorem up to an
additional term $K\log(2)$ appearing in the penalty function.
This small correction is not intrinsic but is required to apply a
union bound over all $K$ in the proof. The main gain is the infimum
over $K$ that ensures that our estimate is almost as good as the one
leading the best single model oracle inequality. Remark that the
infimum is not restricted to $K\leq L$ because any solution with $K>L$
can advantageously be replaced by one with $L$ groups.


\section{Model Selection on a text corpus}
\subsection{NIPS Conference Papers}
In this section, we illustrate our model selection method by an
application to document clustering. Note that we will note use the
constant appearing in Theorem~\ref{thmModelSelectionParticular} but
only the shape of the penalty to select the number of cluster. Note
that the code and the datasets used are available at
\url{https://github.com/EstherBoc/SourceCodeDocumentClustering}.

The dataset "NIPS Conference Papers 1987-2015 Data Set" contains distribution of words used in NIPS conference papers published from 1987 to 2015 \citep{PerJenSpaTeh2016}. Based on a dictionary of $11463$ unique words that appear in $5811$ conference papers, the dataset is represented as a matrix with $11463$ rows and $5811$ columns. Each row represents the number of occurrences of the corresponding word in each document. Problems of dimensionality and low performance are avoided by removing rare words \citep{RigCapYvo2006}. 

To allow for better distinction between clusters, words that appear in more that $80\%$ of the documents are removed. Only $B = 300$ most frequent words are finally considered in the remaining counting matrix. Documents that are empty with respect to this reduced dictionary are removed as well, which leads to $L = 5804$ text documents in the exploited dataset. In our analysis, the vocabulary is constructed a priori from the most frequent and discriminative words and stays fixed. In the supplementary material, model selection with varying dictionary size is addressed in a case where the set of words is assumed to be ordered. In practice, the whole vocabulary is considered there so that different models can be compared to each other. However, the proportions of the first $B$ words can be different, whereas all the others are uniform on the remaining probability. 

The objective is to calibrate the penalty and select the best mixture model, that is, a number of clusters $\widehat{K}$ that has a good bias-variance tradeoff: as too many components may result to an over-fitting, a mixture with too few components may be too restrictive to approximate well the mixture underlying the data. Although time structure of the corpus is ignored in our model selection, it is analyzed in Section \ref{backMap}. One could also consider the corpus as a spatial mixture model with mixture proportions modeling time structure in the articles \citep{CohLep2014}.

\subsection{Clustering algorithm}
Mixture parameters are estimated thanks to the expectation-maximization algorithm. It heavily depends on its initial parameters, so the log-likelihood often converges to a local maximum. We ran 500 short expectation-maximization algorithms from random initializing parameters to analyze the sensitivity of the log-likelihood with respect to the initialization. Figure \ref{initialParamDensity} shows that in practice, despite high model dimension (9000 in this case), the log-likelihood keeps the same order of magnitude with high probability. Although some runs perform poorly, a high majority of them result to the same order of log-likelihood. Thus, a natural way of avoiding local maxima is to run several short expectation-maximization algorithms (with 15 iterations by default) from randomly chosen initializing parameters and run a long one from the most performing parameter in terms of likelihood \citep{TouGas2009}.

\begin{figure}[!!h]
\centering
\includegraphics[scale=0.45]{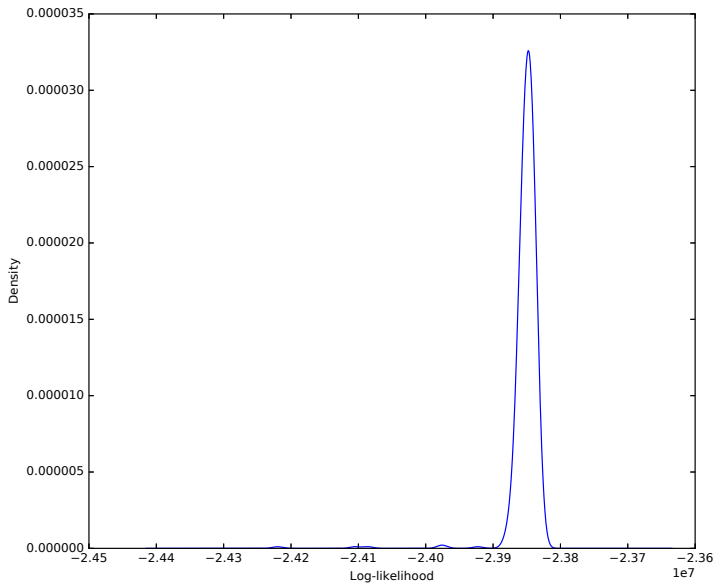}
\captionsetup{justification=centering,margin=2cm}
\caption{\label{initialParamDensity} Kernel Density Estimation of log-likelihood distribution based on 500 random initializing parameters after 15 iterations of EM (K = 30, B = 300)}

\end{figure}

The  expectation-maximization algorithm has another issue: it may converge to the boundary of the parameter space, in which case the output estimates are unstable and some mixture proportions become under-represented \citep{FigJai2002}. This especially occurs in high dimension. B. Zhang, C. Zhang and Xing Yi \cite{ZhaZhaYi2004} proposed a \emph{split and merge} method that divides or merges clusters according to their entropy information but despite random initializations, splitting may decrease stability of the estimates. 

It has been developed in \cite{FigJai2002} that component annihilation leads to more robust results. This observation introduced the expectation-maximization-minimum-message-length algorithm (EM-MML) that starts from a high number of clusters \citep{SilCarFig2014}. At each iteration, mixture proportions are penalized, and if one of them goes under zero, the corresponding cluster is annihilated. The procedure continues until an optimal number of clusters is reached. By simultaneously dealing with the number of components and the estimates, this technique avoids time-consuming computation. However, besides keeping some initialization dependency, the EM-MML penalizes parameters according to a $\BIC$-type function that leads to non robust results in a non asymptotic framework. Further work of Yan, Lai and Lin \cite{YanLaiLin2012} considers all the data set as an initializing parameter. The algorithm progressively annihilates clusters that give decreasing information with respect to Shannon entropy. Although less sensitive to initialization, this technique is time-consuming and based on a penalty that does not take data dimension into account. 

We address a robust expectation-maximization algorithm (EM) that takes these issues into account. It proceeds as follows. For each model, clusters with low proportion estimates are annihilated while other parameters remain the same. After renormalization, additional EM are initialized from the remaining parameters and run until no low proportion mixture is met. This prevents the EM algorithm from approaching the boundary of the parameter space. It thus leads to more robust results in a sense that a 100 component mixture often performs as well as a well-balanced 50 component mixture. Algorithm \ref{robustEM} shows the detailed pseudocode of the algorithm. We remarked that despite annihilation, the number of remaining clusters still increases with respect to the number of components initially input. Several models with different dimensions can thus be explored for penalty calibration.

\begin{algorithm}
\caption{Robust EM algorithm}
\renewcommand{\algorithmicrequire}{\textbf{Input:}}
\renewcommand{\algorithmicensure}{\textbf{Output:}}
\label{robustEM}

\begin{algorithmic}[1]
\Require $k_{\max}$
\Ensure $k_{\text{opt}} ,  \widehat{ \pi}^{(t)}_{i_{\max} }, \widehat{f}^{(t)}_{i_{\max} },  \widehat{\mathcal{L}}^{(t)}_{i_{\max} } $

\Procedure{robustEM}{}

\State $t \gets 0$
\State  $k_{\text{current}} \gets k_{\max}$
\State $\pi_{\text{threshold}} \gets \frac{1}{100 \cdot  k_{\text{current}}}$
\State $\mathcal{L}^{(0)}  \gets -\infty$
\For {i = 1 to 15}
\State $\pi^{(0)}_i, f^{(0)}_i \gets  \Call{initialize}{\pi, f}$
\State $\widehat{\pi}^{(10)}_i , \widehat{f}^{(10)}_i, \widehat{\mathcal{L}}^{(10)}_i  \gets \Call{shortEM}  {k_{\text{current}}, \pi^{(0)}_i, f^{(0)}_i }$
\EndFor

\State $t \gets 10$
\State $i_{\max} \gets \argmax_{i}  \widehat{\mathcal{L}}^{(10)}_i$

\While {$\widehat{\pi}^{(t)}_{i_{\max} } < \pi_{\text{threshold}}$}
\State $k_{\text{current}} \gets \sharp \{ k : \widehat{\pi}^{(t)}_{i_{\max} } (k) \geq \pi_{\text{threshold}}  \} $
\State $\pi_{\text{threshold}} \gets \frac{1}{100 \cdot  k_{\text{current}}}$
\State $\widehat{\pi}^{(t + u)}_{i_{\max} } , \widehat{f}^{(t + u)}_{i_{\max} }, \widehat{\mathcal{L}}^{(t + u)}_{i_{\max} } \gets \Call{EM}{k_{\text{current}}, \pi^{(t)}_{i_{\max} }, f^{(t)}_{i_{\max} }}$
\State $t \gets t + u$
\EndWhile
\State $k_{\text{opt}} \gets k_{\text{current}}$
\EndProcedure

\end{algorithmic}
\end{algorithm}

\subsection{Slope heuristics}
Besides parameter estimation, we want to calibrate our theoretical penalty for selecting a model according to the penalized criterion thus obtained. 

Penalized log-likelihood enables to select a model that has the right
bias-variance tradeoff. Unlike $\AIC$ and $\BIC$ functions, the penalty
introduced in Theorem \ref{thmModelSelectionParticular} is adapted to
a non-asymptotic framework. However, it is defined up to an unknown multiplicative constant. In fact, any greater penalty satisfies the oracle inequality but may lead to a model with high bias. Through slope heuristics, \cite{BauMauMic2010} provide a practical technique to calibrate the constant that leads to an optimal penalty. It relies on the fact that the empirical contrast of the estimated parameter is linear with respect to the model dimension when the model is complex enough. Indeed, for most complex models, the bias term becomes stable so the risk behaves as the variance term and becomes linear with respect to the dimension. Denoting by $\lambda_{\text{min}}$ the slope of the linear part of the empirical contrast, the optimal penalty function is
\begin{equation*}
\mathbf{pen}_{\text{opt}}(K) = 2 \lambda_{\text{min}} D_K,
\end{equation*}
with $D_K$ the model dimension. The derivation of this formula is detailed in \cite{MauMic2008} and \cite{BauMauMic2010}. 

We use slope heuristics to calibrate penalty for NIPS data. Practical methodologies and visualizing graphics were implemented thanks to the R package \textbf{capushe}. The maximum number of clusters in the collection is set to $K_{\text{max}} = 100$. Linear regression is operated with different number of points from which we obtain different slope coefficients. The technique used to deduce the minimal constant is described in \cite{BauMauMic2010}. Figure \ref{slopeHeur} shows linear regression of the log-likelihood with respect to the selected number of points. It gives a slope of $\widehat{\lambda}_{\text{min}} \approx 15$. Finally, the model that minimizes the resulting criterion corresponds to $\widehat{K} = 31$ clusters.

\begin{figure}[!!h]
\centering
\includegraphics[scale=0.5]{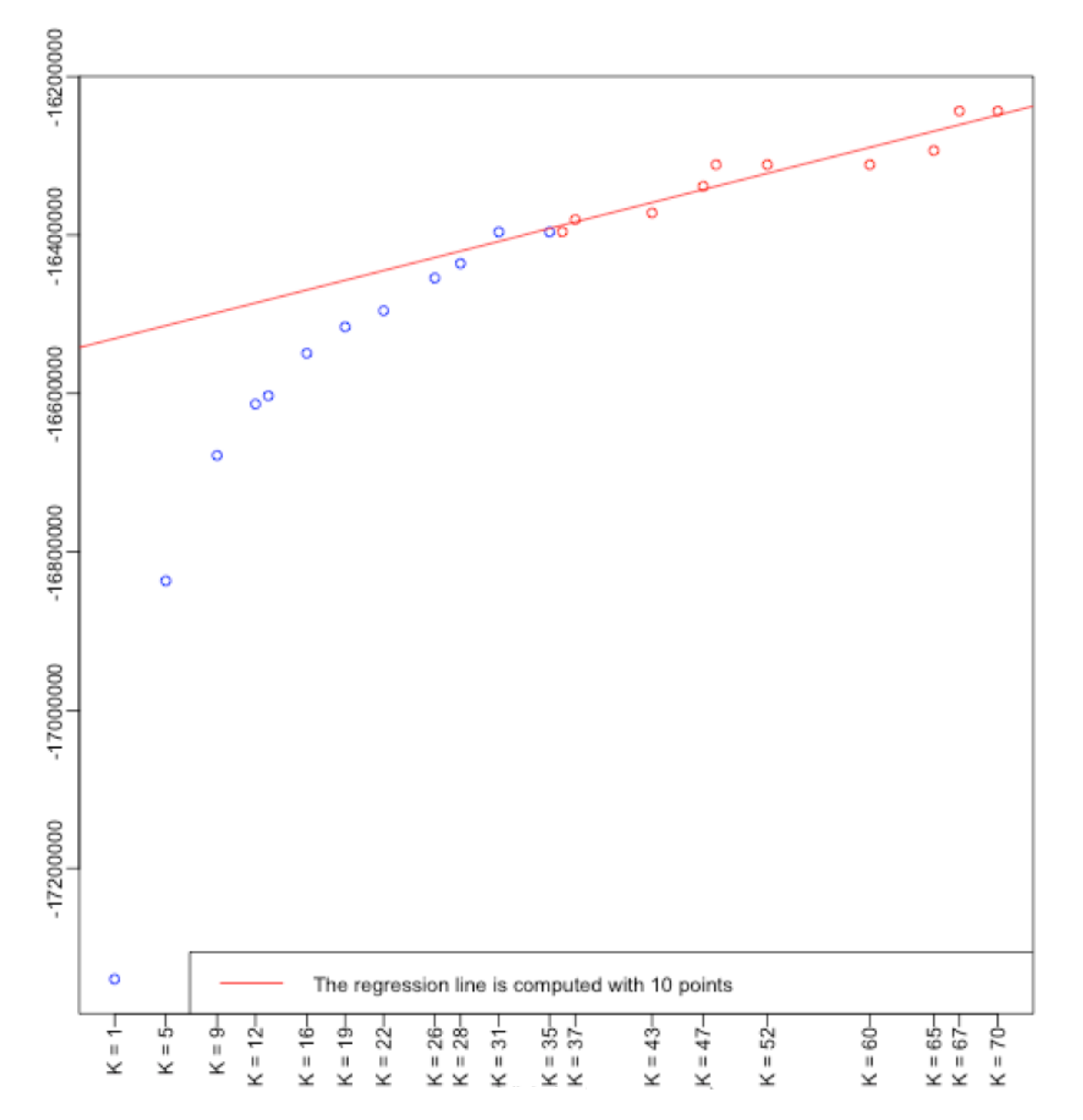}
\captionsetup{justification=centering,margin=2cm}
\caption{\label{slopeHeur} Slope heuristics on NIPS data}
\end{figure} 

\subsection{Back-mapping of NIPS topics}
\label{backMap}
We used the selected model of $\widehat{K} = 31$ clusters and its corresponding parameter estimates to analyze the evolution of some representative NIPS topics overtime. Topics are characterized by their distribution of words. Thus, each cluster can be labeled by looking at the most likely words.

At each year, the average posterior probability of the corresponding articles over every cluster gives a year-scaled evolution of topics from 1987 to 2015. Figures \ref{cluster30} and \ref{topicEvolution} show time evolution of some representative topics and their most likely words.  

The word-cloud in Fig. \ref{cluster30wc} represents the categorical distribution of words. The bigger is a word, the higher its probability is. Clusters become easily interpretable thanks to this representation. Based on its biggest words, we can deduce that Cluster 30 corresponds to reinforcement learning. This topic appeared in the late 80s and knew some ups and downs from then on (Fig. \ref{cluster30}). In Fig. \ref{topicEvolution}, only the most probable words are represented. Figure \ref{cluster0} indicates that the corresponding topic is Bayesian inference and that it constantly increased overtime. So did the popularity of optimization problems (Fig. \ref{cluster18}) and sparsity (Fig. \ref{cluster24}). Classification and support vector machines in particular got most popular in the late 90s before going through a decline (Fig. \ref{cluster8}). Clustering methods knew the same trend, except that their popularity decreased later on. A constant decline of popularity can also be noticed regarding neural networks architecture (Fig. \ref{cluster19}). This could be explained by the fact that the approach and vocabulary on neural networks clearly evolved since 2011, leading to deep learning and its new vocabulary. These trends give a clear overview of how NIPS conferences evolved since 1987 and are similar to those obtained by \cite{PerJenSpaTeh2016} who use Poisson random fields to model the dynamic of the documents' features.


\section{Discussion}
In this paper, we have considered the problem of estimating several
distributions of different-sized categorical variables. We have been
able to prove that the number of components can be estimated almost
optimally by a penalized maximum likelihood principle. The theoretical
analysis has been completed by a numerical illustration on text
clustering in which we have proposed a better initialization scheme
for the expectation-maximization algorithm used to estimate the parameters as well as a
calibration of the penalty.

The results obtained with categorical variables can directly be extended to
bounded continuous ones. Consider a family of $L$ independent random vectors
$(X_1,\ldots,X_L)$ where each $X_l$ represents $n_l$ independent and
identically distributed instances of a random variable that has $s_l$
as a true continuous density distribution. Under the same theoretical
assumptions as those addressed in our work, the same kind of results
can be obtained if we approximate each density by lower bounded
piecewise constant functions. The maximum likelihood estimator on a
model of densities that are piecewise constant on a given partition
would then correspond to a histogram estimator and the objective
would be to estimate $K$ clusters and the best corresponding
histograms. In \cite{Mas2007}, the problem of estimating the best
partition when observing independent and identically distributed
variables is addressed. The theoretical results that are obtained
there could be extended to conditionally independent variables that
are clustered according to their densities' similarities. However, the lower-bound on the model densities as stated in Assumption \ref{hypSurModele} can be no longer applicable 
if densities can take extreme values (Gaussian density for example), because these cannot be controlled in the estimation.

\begin{figure}[!ht]
     \subfloat[Time evolution of reinforcement learning\label{cluster30}]{%
       \includegraphics[width=0.5\textwidth]{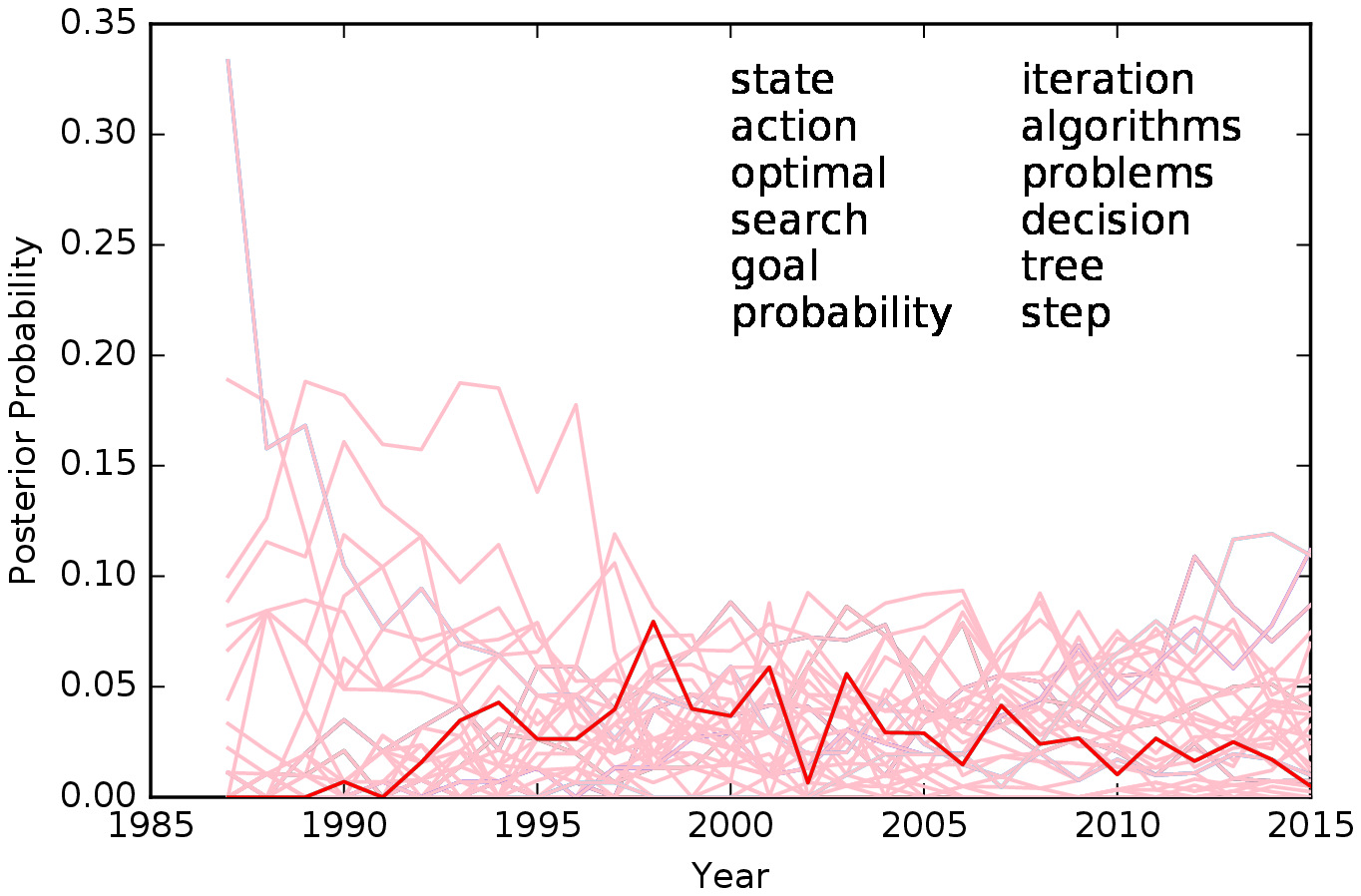}
     }
     \hfill
     \subfloat[Word-cloud of reinforcement learning topic\label{cluster30wc}]{%
       \includegraphics[width=0.45\textwidth, height = 4.8cm]{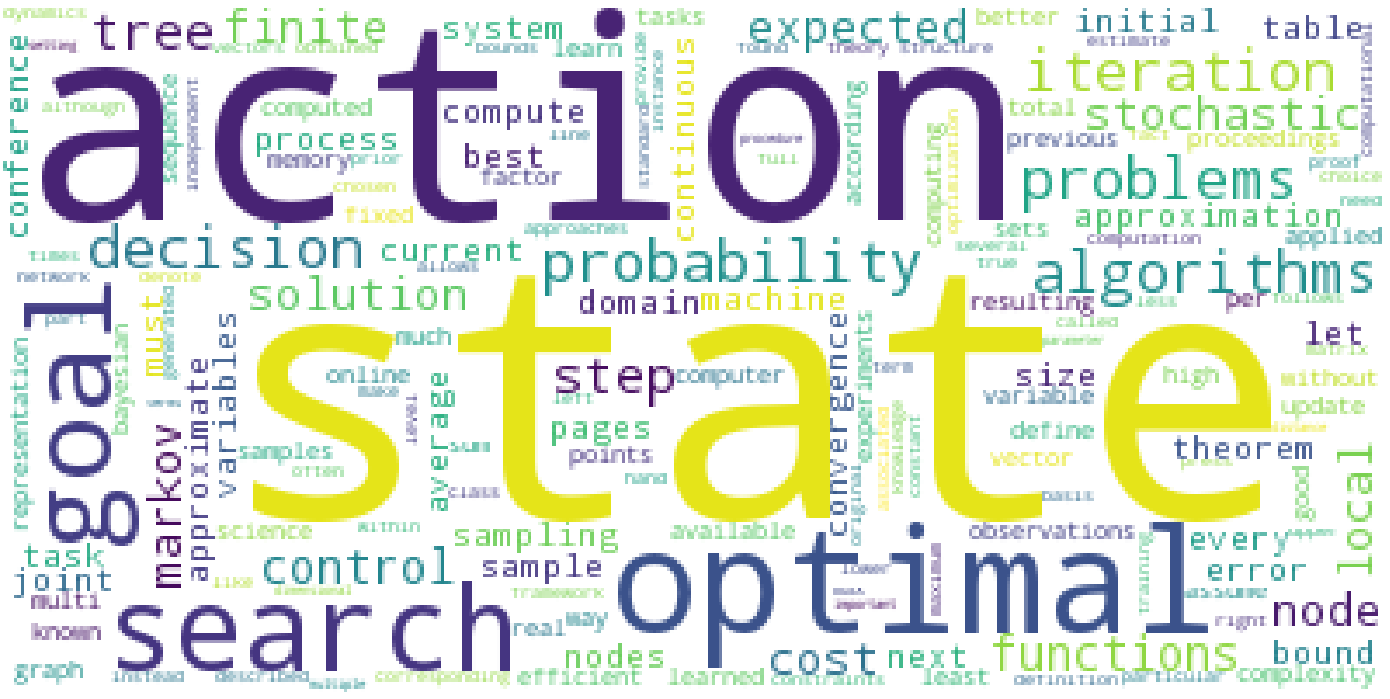}
       }
     
     \caption{Time evolution and word-cloud of reinforcement learning}
     \label{cluster30}
\end{figure}

\begin{figure}[!ht]
     \subfloat[Bayesian inference\label{cluster0}]{%
       \includegraphics[width=0.55\textwidth]{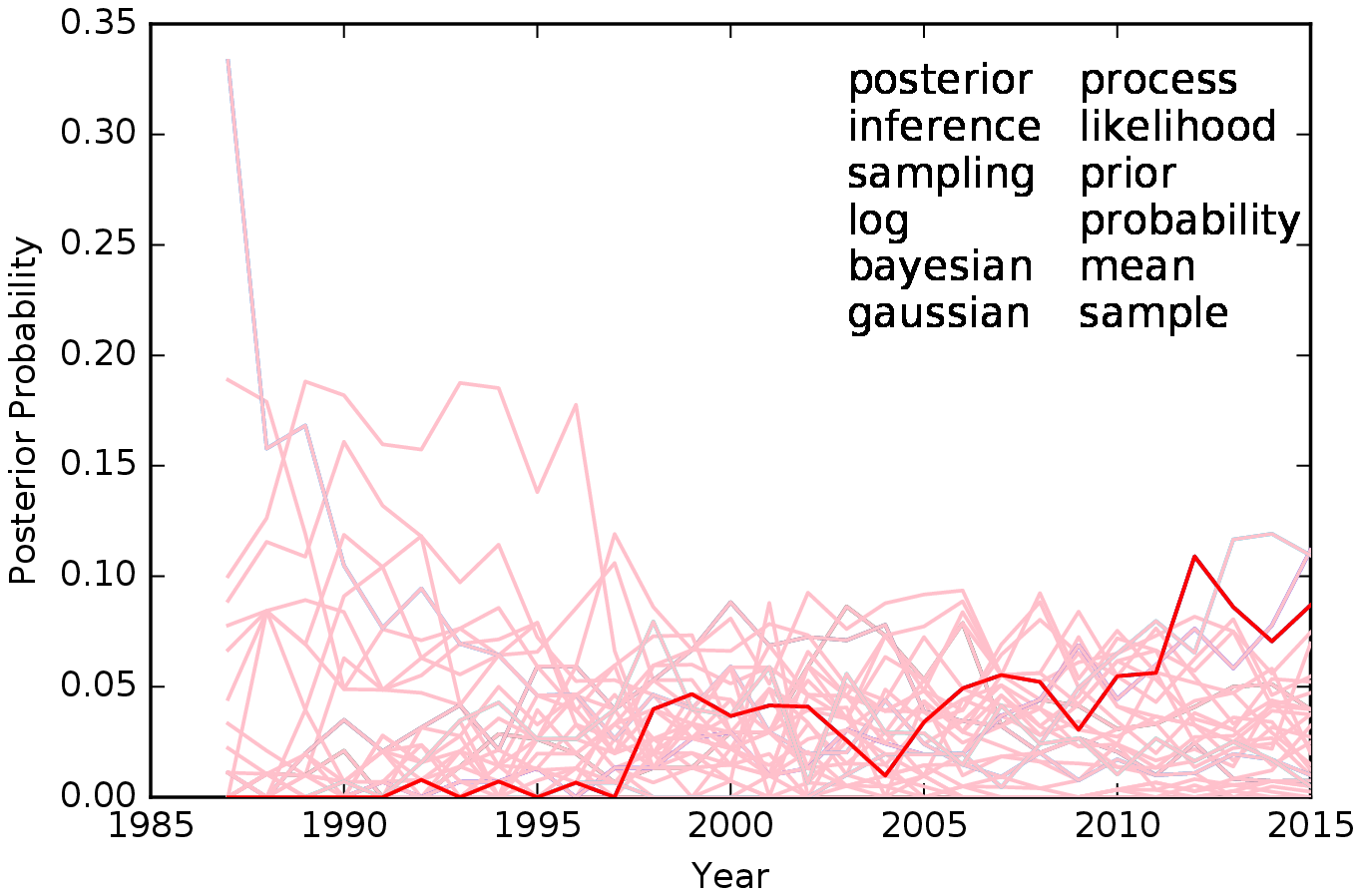}
     }
     \hfill
     \subfloat[Graph Clustering\label{cluster2}]{%
       \includegraphics[width=0.55\textwidth]{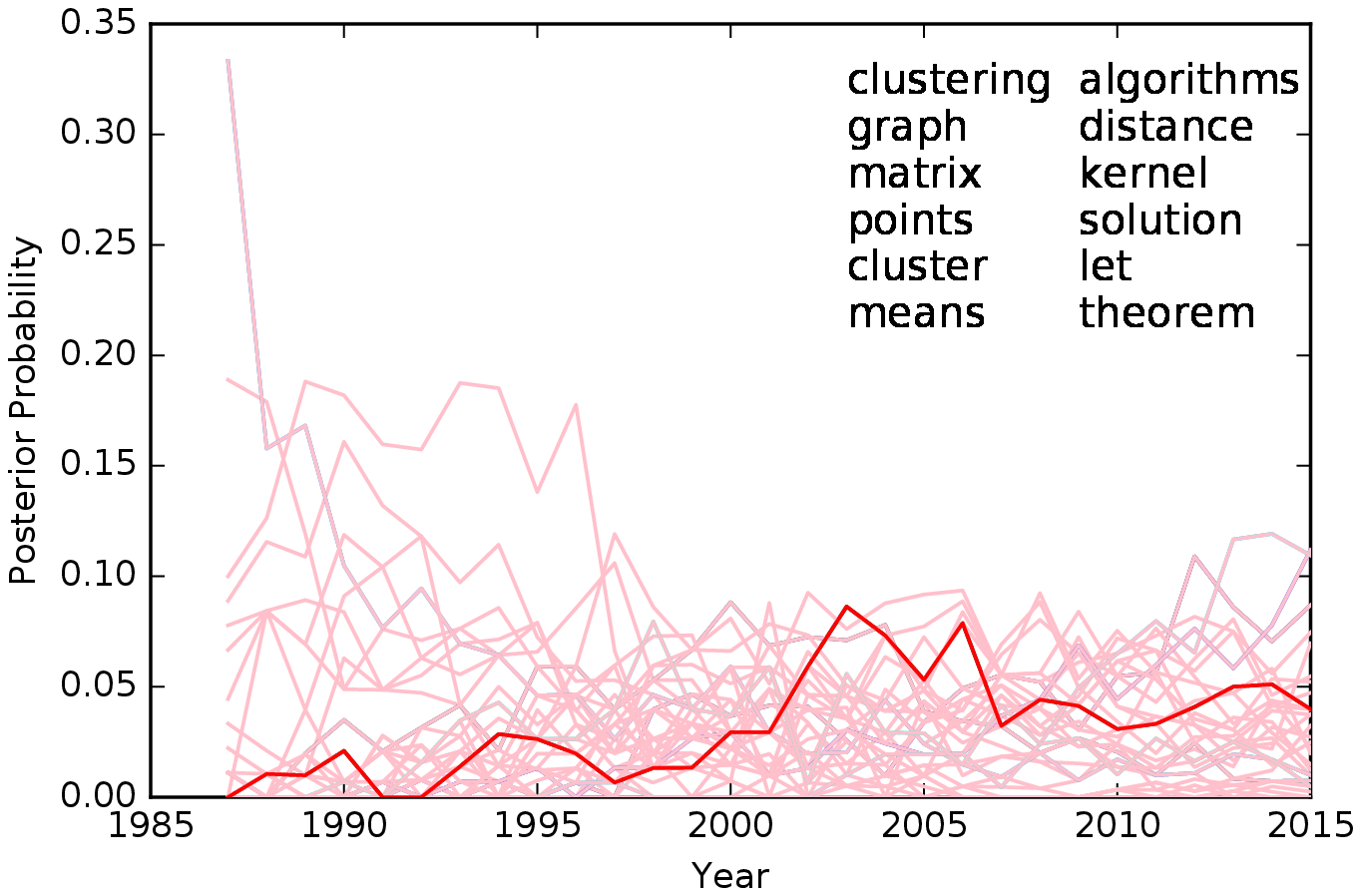}
     }
     \hfill
     \subfloat[Support vector machines\label{cluster8}]{%
       \includegraphics[width=0.55\textwidth]{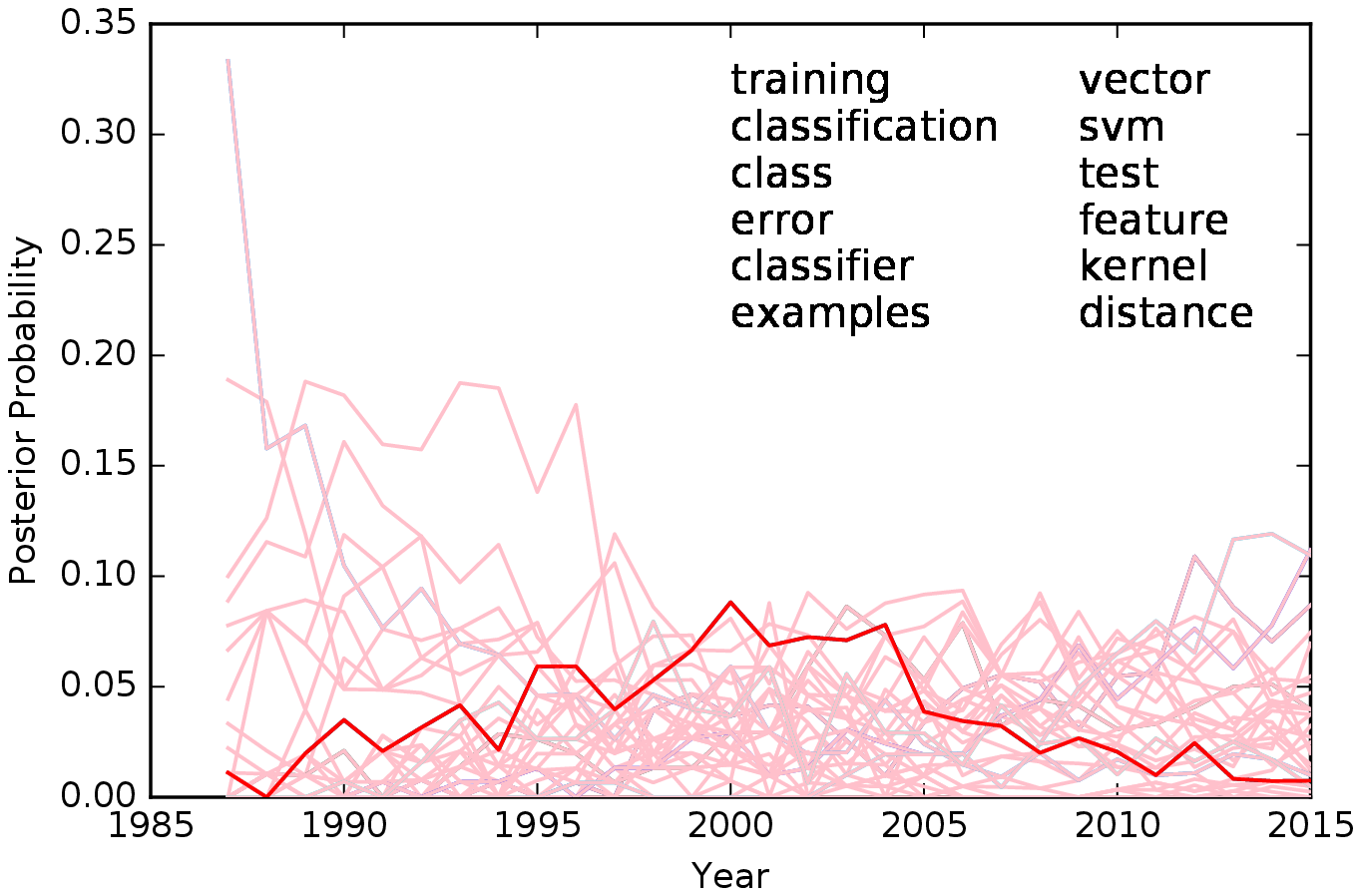}
     }
     \hfill
     \subfloat[Optimization problems\label{cluster18}]{%
       \includegraphics[width=0.55\textwidth]{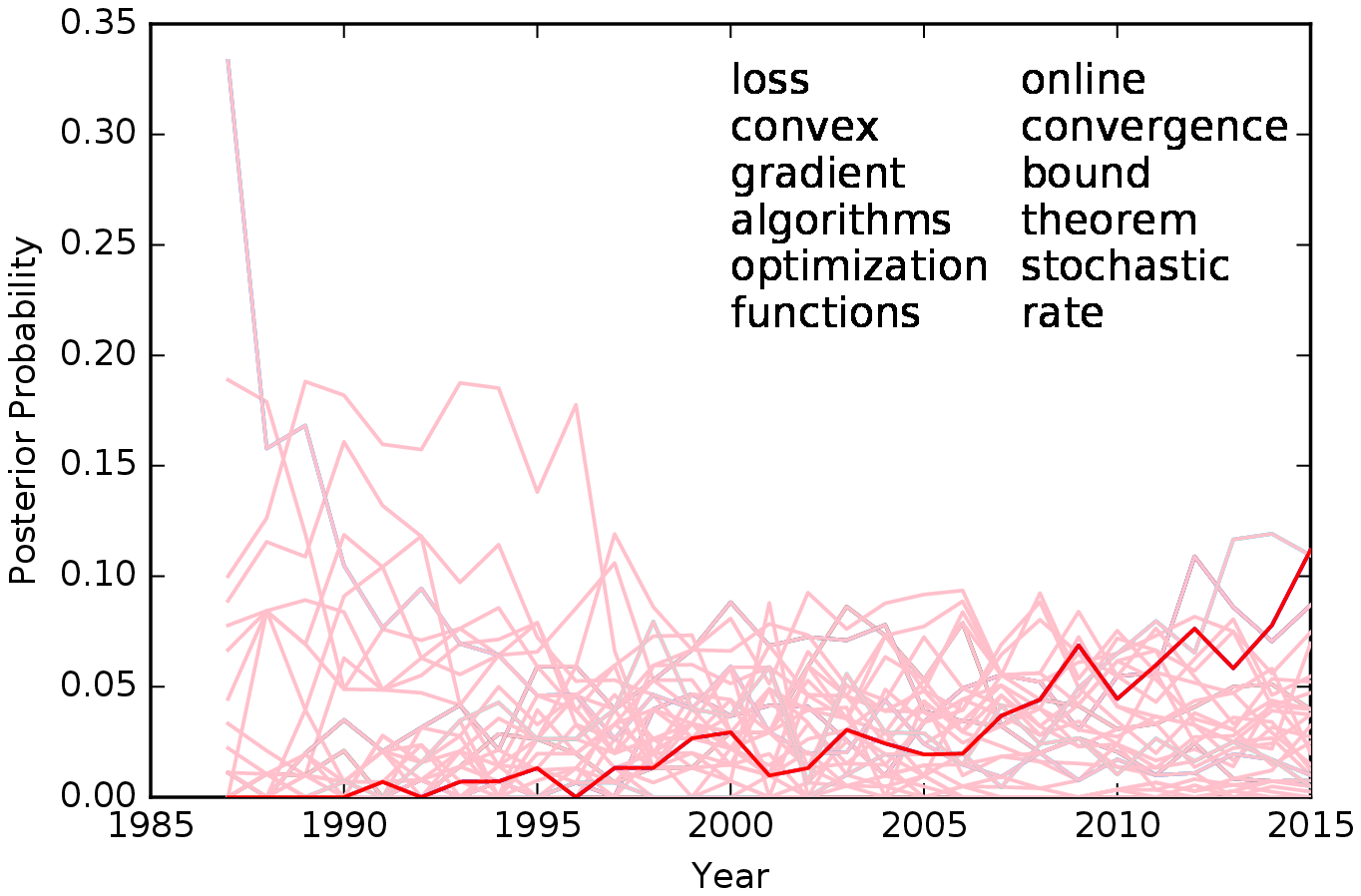}
     }
     \hfill
     \subfloat[Neural Networks\label{cluster19}]{%
       \includegraphics[width=0.55\textwidth]{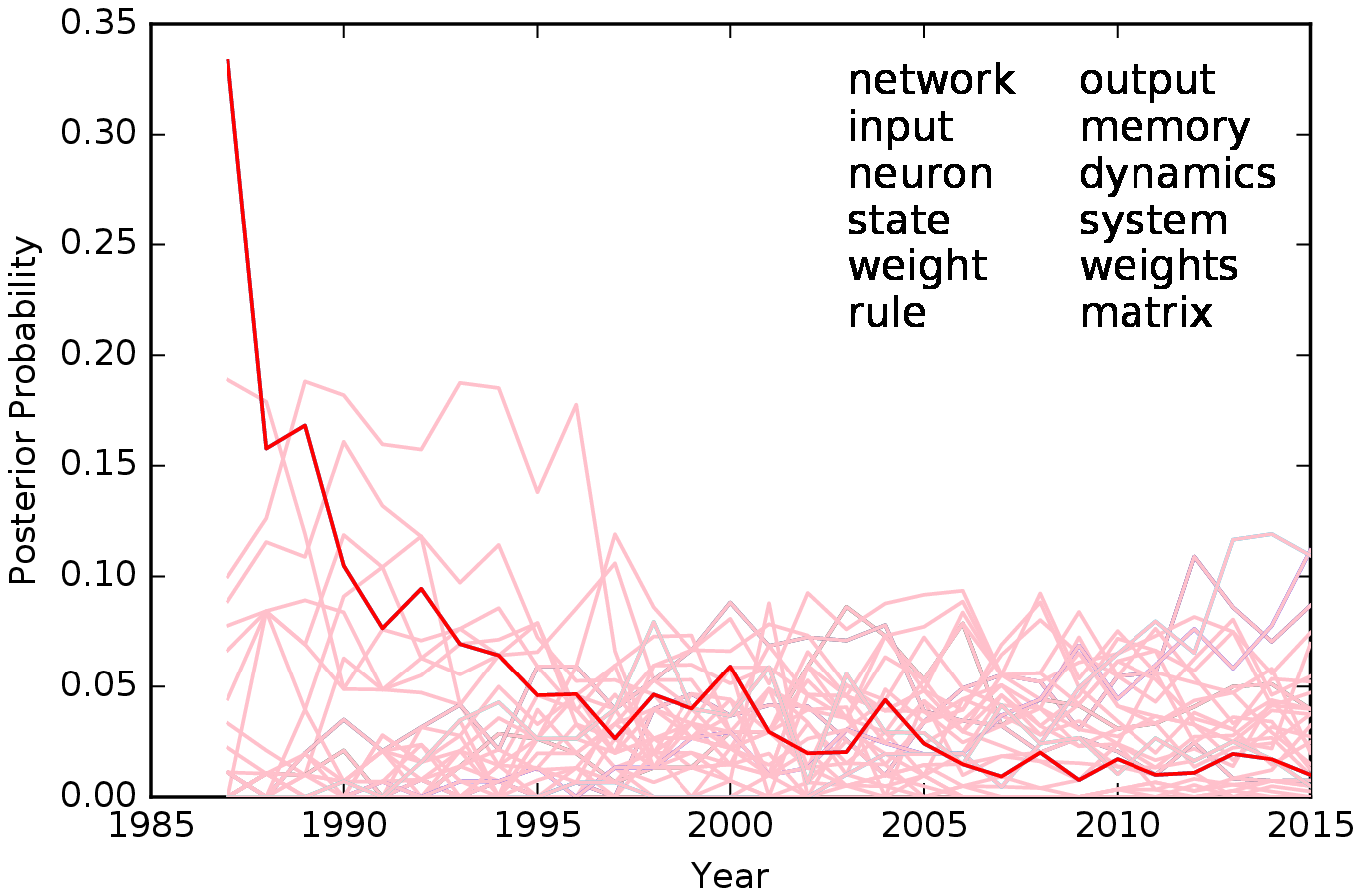}
     }
     \hfill
     \subfloat[Sparsity\label{cluster24}]{%
       \includegraphics[width=0.55\textwidth]{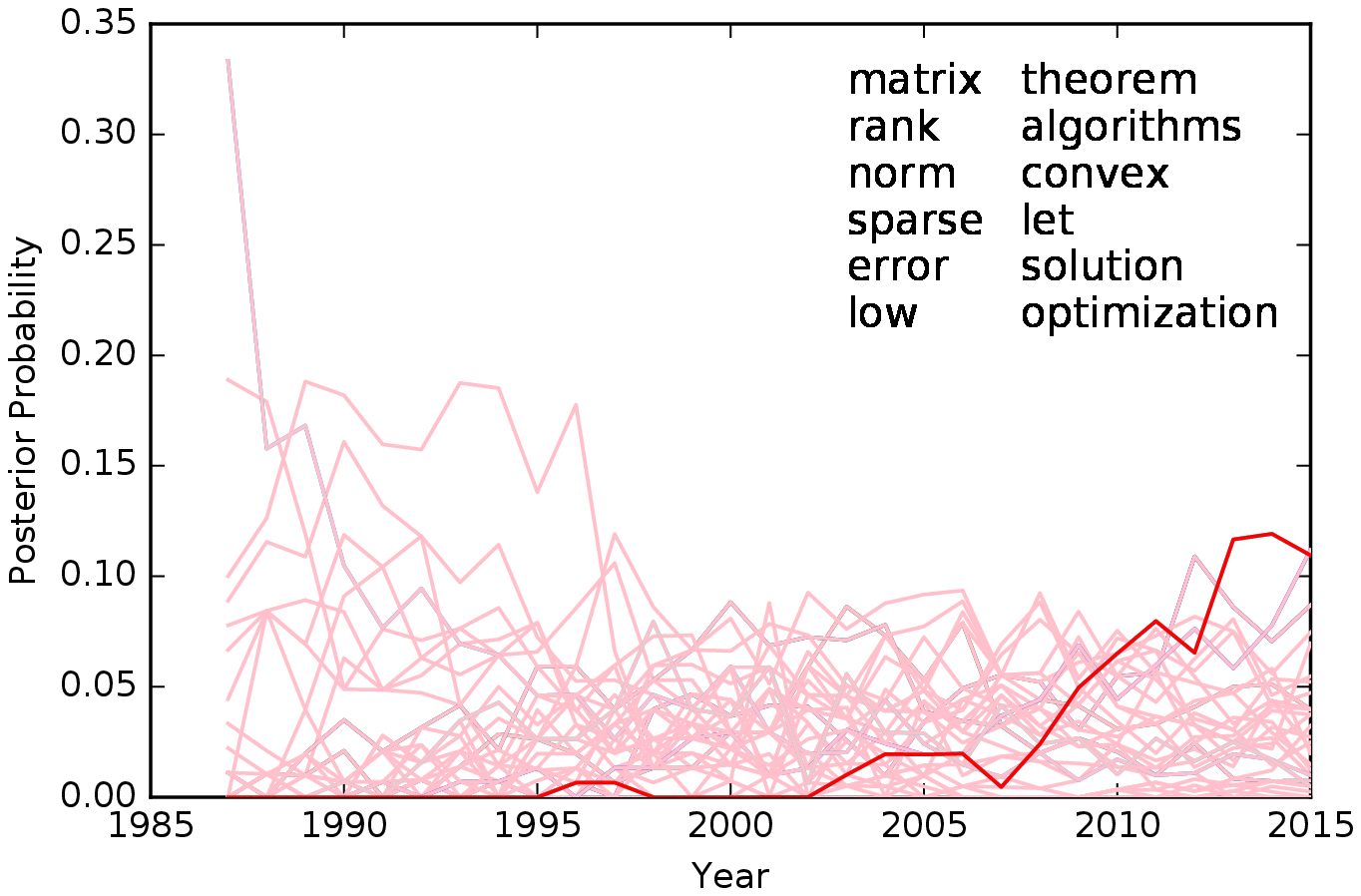}
     }
     
     \caption{Comparative evolution of some representative NIPS topics from 1987 to 2015 with the 12 most likely words. The red line corresponds to the labeled topic and the transparent one to the other topics.}
     \label{topicEvolution}
\end{figure}

\clearpage

\bibliography{BiblioPenalizedLL}
\bibliographystyle{apalike}

\newpage
\appendix
\section{Supplementary material}


\subsection{Preliminary}
In this section, Theorems \ref{thmSingleModelParticular} and \ref{thmModelSelectionParticular} are proved in a more general setting where some assumptions are relaxed. Notation is generalized so that each model is indexed by $m$ rather than $K$. Thus, a parameter set of a model is denoted by $\Theta_m$, the corresponding model by $\mathcal{S}_m$ and $\mathcal{F}_m$ is the associated $K$-uplet of mixture distributions. An element $f\in \mathcal{F}_m$ can also be described as a family $(f_1,\ldots,f_K)$ of $K$ categorical probability distributions on $B$ categories.
Denoting by $s_l$ the true density of observation $l$, we assume the structure of the models to be such that:
\begin{assumption}[\textbf{Model structure}]
\label{GhypSurModele}
Any density functions $f, f' \in \mathcal{F}_{m}$ taken in a model satisfy $e^{-\tau_n} \leq f/f' \leq e^{\tau_n}$ and $e^{-\tau_n} \leq f/s_l$, $l\in \{1,\ldots,L\}, s_l$ being the true density of observation $l$.
\end{assumption}
This assumption implies Assumption \ref{hypSurModele} as formulated inside the article. Moreover, the second part of the assumption is made on the model density rather than on the true one which is unknown.
Our main result adresses an oracle inequality that links some averaged Kullback-Leibler divergence \textbf{KL} of the selected estimator to the averaged \textbf{KL} between the true densities and every model within the collection. This inequality leads to some penalty function which heavily relies on a notion of bracketing entropy. A bracket $[f_k^-, f_k^+]$ is a pair of real-valued functions such that for all $x\in \mathcal{X}$, $f_k^-(x) \leq f_k^+(x)$. A density function $f_k$ is said to belong to the bracket $[f_k^-, f_k^+]$ if $f_k^-(x) \leq f_k(x)\leq f_k^+(x)$ for all $x\in \mathcal{X}$. Take ${f}^{-}, {f}^{+}  \in\mathcal{F}_{m}$ such that for all $k= 1,\ldots, K$, $[{f^-_k}, {f^+_k}]$ forms a bracket. Fix $(k_1,\ldots,k_L)$ a cluster assignment of the observations. We define the width of such a family of brackets as follows:
\begin{equation*}
\mathbf{a}({f}^{-}, {f}^{+}) = \frac{1}{n} \sum_{l=1}^{L}n_lE_{s_l}\left[  \mid \log\left( \frac{{f^-_{k_l}}}{{f^+_{k_l}}}\right)\mid^2 \right].
\end{equation*}
The bracketing entropy $H_{[.], \mathbf{a}}(\delta, \mathcal{F}_m)$ of a set of functions $\mathcal{F}_m$ is defined as the logarithm of the minimum number of brackets of width smaller than $\delta$ such that every function of $\mathcal{F}_m$ belongs to one of these brackets. The model complexity that will be considered rather depends on the localized models, which in our framework can be written as:
\begin{equation*}
\mathcal{F}_{m}(\tilde{f}, \sigma) = \left\{f \in \mathcal{F}_{m} \mid \mathbf{a}(\tilde{f}, f) \leq \sigma^2\right\}.
\end{equation*} 
We also impose a structural assumption on the localized models:
\begin{assumption}
\label{hypPhi}
There exists a real-valued function $\phi_m$ on $[0, +\infty)$ such that $\phi_m$ is non-decreasing, the mapping $\delta \mapsto \frac{1}{\delta}\phi_m(\delta)$ is non-increasing on $(0, +\infty)$
and for every $\sigma\geq 0$ and every $f \in \mathcal{F}_m$,
\begin{equation*}
\int_{0}^{\sigma}H_{[.], \mathbf{a}}(\delta, \mathcal{F}_m(f, \sigma))^{1/2}d\delta \leq \phi_m(\sigma).
\end{equation*}
\end{assumption}
The use of divergence \textbf{a} instead of \textbf{KL} has some benefit, as it allows to take advantage of the metric entropy of the models and deduce the bracketing entropy. It is smaller, up to a certain constant, than the \textbf{KL} divergence:
\begin{proposition}
\label{meynet}
Assume that Assumption \ref{GhypSurModele} is satisfied. Then
\begin{equation*}
\frac{1}{n} \sum_{l=1}^{L}n_lE_{s_l}\left[    \mid \log\left( \frac{f_{k_l}}{s_l}\right)\mid^2  \right] \leq \frac{\tau_n^2}{e^{-\tau_n} + \tau_n -1}\frac{1}{n}\sum_{l= 1}^{L}n_l \mathbf{KL}(f_{k_l}, s_l).
\end{equation*}
\end{proposition}

\begin{proof}
The proof can be deduced from Meynet's result \citep{Mey2012} that states the following.
\begin{lemma}
Let $P$ and $Q$ be two probability measures with $P\ll Q$. Assume there exists $\tau>0$ such that $\log(  \lVert \frac{dP}{dQ} \rVert _\infty)\leq \tau$. Then
\begin{equation*}
\int \left( \log \frac{dP}{dQ} \right)^2 dP \leq \frac{\tau^2}{e^{-\tau} + \tau -1}\mathbf{KL}(P,Q).
\end{equation*}
\end{lemma}
By taking $\tau= \tau_n$, $dP = f_{k_l}d\mu$ and $dQ = s_ld\mu$, one can deduce the result.
\end{proof}

In order to avoid measurability issues, we also impose a separability condition on the models: 
\begin{assumption}[\textbf{Separability}]
\label{hypSep}
There exists a countable subset $\mathcal{F}'_m$ of $\mathcal{F}_m$ and a set $\mathcal{X}'_m$ of measure $\lambda(\mathcal{X} \setminus  \mathcal{X}'_m) = 0$ such that for every $f \in \mathcal{F}_m$, there exists a sequence $(f_{j})_{j\geq 1}$ of elements of $\mathcal{F}'_{m}$ such that for every $x\in \mathcal{X}'_m$, $\log(f_{j}(x)) $ goes to $\log(f(x))$ as $j$ goes to infinity.
\end{assumption}

\subsection{Single Model Risk Bound}

\begin{theorem}
\label{thmSingleModel}
Let $(x_1,\ldots,x_L)$ be $L$ observations of independent random vectors $(X_1,\ldots,X_L)$ where each $X_l$ consists of $n_l$ independent and identically distributed instances of a multinomial vector that has $s_l$ as a true density with respect to some known positive measure. Assume $\mathcal{S}_m $ is a model for which Assumptions \ref{GhypSurModele}, \ref{hypPhi} and \ref{hypSep}  hold. Let $(\widehat{\pi}, \widehat{f}) \in \Theta_{m}$ be an $\eta$ -log-likelihood minimizer in $\mathcal{S}_m$:
\begin{equation*}
\gamma_n(\widehat{\pi}, \widehat{f})     \leq   \inf_{(\pi ,f )\in \Theta_{m}}    \gamma_n(\pi, f)   + \eta,  
\end{equation*}
and define its corresponding cluster assignment for each observation $x_l$:
\begin{equation*}
\widehat{k_l} = \argmax_{k \in \Iintv{1,K}} \left\{ \widehat{\pi_k} \left(\prod_{i = 1}^{n_l} \widehat{f_k}(x_l^i) \right) \right\}.
\end{equation*}
Then, for any $C_1>1$, there exists a constant $C_2$ depending only on $C_1$ such that, for $\mathfrak{D}_m = n\sigma_m^2$ with $\sigma_m$ the unique root of $\frac{1}{\sigma}\phi_m({\sigma}) = n^{1/2}\sigma$, the estimate $(\widehat{\pi}, \widehat{f})$ satisfies
\begin{equation*}
\begin{split}
E\left[\sum_{l = 1}^{L} \frac{n_l}{n}\mathbf{KL}(s_l, \widehat{f_{\widehat{k_l}}}) \right]   &\leq  C_1  \left\{  \inf_{\substack{f \in \mathcal{F}_m\\(k_l)_l \in \Iintv{1,K}^L}} \left(\sum_{l=1}^{L}\frac{n_l}{n} \mathbf{KL}(s_l, f_{k_l})\right)       + \frac{ (2+\kappa_0) L\log K + \kappa_0\mathfrak{D}_m}{n} + \frac{ \eta}{n} \right\}\\
&\quad+ \frac{C_2}{n} .
\end{split}
\end{equation*}
\end{theorem}

\begin{proof}
By definition of $\widehat{k}_l$ and using the fact that all mixture probabilities $\pi_k$ are less than 1, we can write:
\begin{equation}
\begin{split}
\label{majoreLL}
\eta + \inf_{(\pi ,f )\in \Theta_{m}}    \gamma_n(\pi, f)  &\geq     \sum_{l = 1}^{L}-\log    \left(   K \widehat{\pi}_{\widehat{k}_l} \left[\prod_{i = 1}^{n_l} \left( \frac{\widehat{f}_{\widehat{k}_l}}{s_l}\right) (x_l^i) \right]   \right)   \\
&\geq -L\log (K) - \sum_{l=1}^{L}\sum_{i=1}^{n_l} \log     \left(         \left(\frac{\widehat{f}_{\widehat{k}_l}}{s_l}\right)      (x_l^i)   \right).
 \end{split}
\end{equation}
On the other hand, for any cluster assignment $(k_1,\ldots,k_L)$ of the observations we have
\begin{equation*}
\begin{split}
 \gamma_n(\pi, f)
 &\leq \sum_{l=1}^{L} -\log(\pi_{k_l}) +  \sum_{l=1}^{L} - \log   \left(  \prod_{i = 1}^{n_l} \left( \frac{f_{k_l}}{s_l}\right)(x_l^i)    \right)   
 \end{split}
\end{equation*}
because $-\log$ is non-increasing. Therefore, by inequality (\ref{majoreLL}),
\begin{equation*}
\begin{split}
\inf_{(\pi ,f )\in \Theta_{m}}  \gamma_n(\pi, f) &\leq \inf_{(k_l)_l\in \Iintv{1,K}^L}\left[\inf_{(\pi ,f )\in \Theta_{m}} \left\{\sum_{l=1}^{L} -\log(\pi_{k_l}) +  \sum_{l=1}^{L}\sum_{i = 1}^{n_l}  - \log   \left( \left( \frac{f_{k_l}}{s_l}\right)(x_l^i)  \right)  \right\}  \right],
 \end{split}
\end{equation*}
which leads to 
\begin{equation}
\label{tobeCentralized}
\begin{split}
\sum_{l=1}^{L}\sum_{i=1}^{n_l}  - \log     \left(         \left(\frac{\widehat{f}_{\widehat{k}_l}}{s_l}\right)      (x_l^i)   \right)     &\leq         L\log(K)  + \eta \\
 &\quad+  \inf_{(k_l)_l\in \Iintv{1,K}^L}\left[\inf_{(\pi ,f )\in \Theta_{m}} \left\{\sum_{l=1}^{L} -\log(\pi_{k_l})\right.\right. \\
 &\quad+ \left.\left. \sum_{l=1}^{L}\sum_{i = 1}^{n_l}  - \log   \left( \left( \frac{f_{k_l}}{s_l}\right)(x_l^i)  \right)  \right\}  \right].
\end{split}
\end{equation}
We define by $\bar{f} = (\bar{f}_1,\ldots, \bar{f}_K)$ a family of densities that satisfy for all $\delta >0$ and all cluster assignment $(k_1,\ldots, k_L)$:
\begin{equation}
\label{fbarIneq}
\sum_{l = 1}^{L} n_l \mathbf{KL}(s_l, \bar{f}_{k_l}) \leq \inf_{f \in \mathcal{F}_m}\sum_{l=1}^{L} n_l \mathbf{KL}(s_l, f_{k_l}) +\delta.
\end{equation}
For all $l \in \{1,\ldots, L\}$ and any density $f_{k_l}$ at $l$, we denote the empirical Kullback-Leibler divergence $\mathbf{kl}(f_{k_l})$ by:
\begin{equation*}
\mathbf{kl}(f_{k_l}) = \sum_{i = 1}^{n_l}-\log    \left(\frac{f_{k_l}}{s_l}\right).
\end{equation*} 
We have $E\left[\sum_{l = 1}^{L} \mathbf{kl}(f_{k_l})(X_l)\right] = \sum_{l = 1}^{L} n_l \mathbf{KL}(s_l, f_{k_l})$. Eventually, define:
\begin{equation}
\label{defCentralized}
\nu_{l}(f_{k_l}) = \mathbf{kl}(f_{k_l})(x_l) - E_{s_l}\left[\mathbf{kl}(f_{k_l})(X_l)\right]
\end{equation} 
the centered version of the empirical \textbf{KL}-divergence. Thanks to equation \eqref{tobeCentralized} and using the definition of $\bar{f}$ in (\ref{fbarIneq}), we can write:
\begin{equation*}
\begin{split}
\sum_{l = 1}^{L}\nu_l(\widehat{f}_{\widehat{k_l}}) + \sum_{l = 1}^{L} E_{s_l} \left[\mathbf{kl}(\widehat{f}_{\widehat{k_l}})(X_l)\right]      &\leq      L\log K +  \inf_{(k_l)_l\in \Iintv{1,K}^L}   \left\{ \inf_{\pi \in \mathbb{S}_{k-1}}   \left( \sum_{l=1}^{L} -\log(\pi_{k_l}) \right)\right.\\
&\quad+ \left.  \sum_{l=1}^{L}\left(  \sum_{i = 1}^{n_l} - \log   \left( \left( \frac{\bar{f}_{k_l}}{s_l}\right)(x_l^i)    \right) \right.  -   E_{s_l}\left[\sum_{i = 1}^{n_l} \log   \left( \left(\frac{\bar{f}_{k_l}}{s_l}\right) (X_l^i)   \right)\right]    \right)\\    
&\quad+ \left. \sum_{l=1}^{L} E_{s_l}\left[\sum_{i = 1}^{n_l} \log   \left( \left(\frac{\bar{f}_{k_l}}{s_l}\right)(X_l^i)   \right)\right] \right\}+ \eta \\
&\leq  L\log K +    \inf_{(k_l)_l\in \Iintv{1,K}^L}   \left\{ \inf_{\pi \in \mathbb{S}_{k-1}}   \left( \sum_{l=1}^{L} -\log(\pi_{k_l}) \right)+ \sum_{l = 1}^{L} \nu_l(\bar{f}_{k_l}) \right. \\
&\quad \left.+ \sum_{l = 1}^{L} n_l \mathbf{KL}(s_l, \bar{f}_{k_l})  \right\}+ \eta
\end{split}
\end{equation*}
which can be rewritten as:
\begin{equation}
\label{majorationKL}
\begin{split}
\sum_{l = 1}^{L} n_l \mathbf{KL}(s_l, \widehat{f}_{\widehat{k_l}})        &\leq      L\log K +   \inf_{(k_l)_l\in \Iintv{1,K}^L} \left\{  \sum_{l=1}^{L} -\log(\frac{1}{K}) +\sum_{l = 1}^{L} \nu_l(\bar{f}_{k_l})\right.\\
&\quad \left. +  \inf_{f \in \mathcal{F}_m} \left(\sum_{l=1}^{L} n_l \mathbf{KL}(s_l, f_{k_l})\right) \right\} 
-\sum_{l = 1}^{L}\nu_l(\widehat{f}_{\widehat{k_l}})  +\delta  + \eta\\
&\leq 2L\log K +  \inf_{(k_l)_l\in \Iintv{1,K}^L}\left\{  \sum_{l = 1}^{L} \nu_l(\bar{f}_{k_l}) +  \inf_{f \in \mathcal{F}_m} \left(\sum_{l=1}^{L} n_l \mathbf{KL}(s_l, f_{k_l})\right) \right\} \\
&\quad-\sum_{l = 1}^{L}\nu_l(\widehat{f}_{\widehat{k_l}})  +\delta  + \eta,
\end{split}
\end{equation}
where the infimum over the simplex is upper-bounded by a uniform distribution. It remains to get an upper bound of the deviation $-\sum_{l = 1}^{L}\nu_l(\widehat{f}_{\widehat{k_l}})$, which is based on the following lemmas.
First define for all random variable $Z$ and any event $A$ of positive probability, $E^A[Z]= \frac{E[Z\mathbb{1}_A]}{\text{pr}(A)}$. We have the following results.

\begin{lemma}
\label{markovIneq}
Let $Z$ be a random variable and $\Psi$ a non-decreasing function such that for all measurable event $A$ satisfying $\text{pr}(A)>0$, $E^A[Z]\leq \Psi\left(  \log\left(    \frac{1}{\text{pr}(A)}\right)   \right)$. Then, for all $x\geq 0$, $\text{pr}(Z>\Psi(x))\leq e^{-x}$.
\end{lemma}

\begin{lemma}
\label{lemmaConc}
Let $(k_1,\ldots,k_L) \in \Iintv{1,K}^L$ be a group assignment for each observation $x_l$. Then, there exist three absolute constants $\kappa'_0 >4$, $\kappa'_1$ and $\kappa'_2$ such that, under Assumption \ref{hypPhi}, for all $m\in \mathcal{M}$, for all $y_m>\sigma_m$ and every measurable event $A$ such that $\text{pr}(A)>0$,
\begin{equation*}
\begin{split}
E^A \left[  \sup_{f \in \mathcal{F}_m}  \frac{1}{n}\sum_{l = 1}^{L}    \frac{-\nu_l(f_{k_l}) }{y_m^2  +      \kappa'_0 \frac{1}{n}\sum_{l=1}^{L}n_l  E   \left[   \mid    \log\left(\frac{f_{k_l}}{s_{l}}\right) \mid  ^2 \right]  }   \right] 
&\leq \kappa'_1\frac{\sigma_m}{y_m} +  \frac{\kappa'_2} {(ny_m^2)^{1/2}}    \log \left( \frac{1}{\Proba(A)}  \right)^{1/2}\\
&\quad + \frac{9\tau_n}{ny_m^2}    \log \left( \frac{1}{\Proba(A)} \right).
\end{split}
\end{equation*}
\end{lemma}

Then, for all $\lambda >0$ we can derive that:
\begin{equation*}
\begin{split}
E^A \left[ \sup_{(k_l)_l\in \Iintv{1,K}^L}  \sup_{f \in \mathcal{F}_m} \exp\left(      \frac{ \lambda }{n}\sum_{l = 1}^{L}    \frac{-\nu_l(f_{k_l}) }{y_m^2  +        \kappa'_0 \frac{1}{n}\sum_{l=1}^{L}n_l  E   \left[   \mid    \log\left(\frac{f_{k_l}}{s_{l}}\right) \mid  ^2 \right]   }      \right)\right] \\
= E^A \left[ \sup_{(k_l)_l\in \Iintv{1,K}^L}   \exp\left(  \sup_{f \in \mathcal{F}_m}    \frac{ \lambda }{n}\sum_{l = 1}^{L}    \frac{-\nu_l(f_{k_l}) }{y_m^2  +        \kappa'_0 \frac{1}{n}\sum_{l=1}^{L}n_l  E   \left[   \mid    \log\left(\frac{f_{k_l}}{s_{l}}\right) \mid  ^2 \right]   }      \right)\right] \\
\leq E^A \left[ \sum_{(k_l)_l\in \Iintv{1,K}^L}  \exp\left(  \sup_{f \in \mathcal{F}_m}    \frac{ \lambda }{n}\sum_{l = 1}^{L}    \frac{-\nu_l(f_{k_l}) }{y_m^2  +        \kappa'_0 \frac{1}{n}\sum_{l=1}^{L}n_l  E   \left[   \mid    \log\left(\frac{f_{k_l}}{s_{l}}\right) \mid  ^2 \right]   }      \right)\right]  \\
\end{split}
\end{equation*}
Therefore, by Lemmas \ref{lemmaConc} and \ref{markovIneq}, for all $x>0$, except on a set of probability less than $e^{-x}$,
\begin{equation*}
\begin{split}
&\sum_{(k_l)_l\in \Iintv{1,K}^L}   \exp\left(  \sup_{f \in \mathcal{F}_m}    \frac{ \lambda }{n}\sum_{l = 1}^{L}    \frac{-\nu_l(f_{k_l}) }{y_m^2  +        \kappa'_0 \frac{1}{n}\sum_{l=1}^{L}n_l  E   \left[   \mid    \log\left(\frac{f_{k_l}}{s_{l}}\right) \mid  ^2 \right]   }      \right)\\ 
\leq &\sum_{(k_l)_l\in \Iintv{1,K}^L}  \exp \left(     \frac{\lambda\kappa'_1\sigma_m}{y_m} +   \frac{\lambda  \kappa'_2} {(ny_m^2)^{1/2}}       x^{1/2} +\lambda\left( \frac{9\tau_n}{ny_m^2}   \right) x      \right) \\
\leq &K^L  \exp \left(     \frac{\lambda\kappa'_1\sigma_m}{y_m} +   \frac{\lambda  \kappa'_2} {(ny_m^2)^{1/2}}       x^{1/2} +\lambda\left( \frac{9\tau_n}{ny_m^2}   \right) x      \right).
\end{split}
\end{equation*}
A fortiori, we have 
\begin{equation*}
\begin{split}
 \exp\left(      \frac{ \lambda }{n}\sum_{l = 1}^{L}    \frac{-\nu_l(\widehat{f}_{\widehat{k}_l})}{y_m^2  +        \kappa'_0 \frac{1}{n}\sum_{l=1}^{L}n_l E   \left[   \mid    \log\left(\frac{\widehat{f}_{\widehat{k}_l}}{s_{l}}\right) \mid  ^2 \right]   }      \right) 
\leq    K^L\exp \left(     \frac{\lambda\kappa'_1\sigma_m}{y_m} +   \frac{\lambda  \kappa'_2} {(ny_m^2)^{1/2}}       x^{1/2} +\lambda\left( \frac{9\tau_n}{ny_m^2}   \right) x      \right), 
\end{split}
\end{equation*}
and except on a set of probability less than $e^{-x}$, 
\begin{equation*}
\begin{split}
    \frac{1}{n}\sum_{l = 1}^{L}    \frac{-\nu_l(\widehat{f}_{\widehat{k}_l})}{y_m^2  +      \kappa'_0\frac{1}{n}\sum_{l=1}^{L}n_l E   \left[   \mid    \log\left(\frac{\widehat{f}_{\widehat{k}_l}}{s_{l}}\right) \mid  ^2 \right]  }    
\leq  \frac{ L \log(K) } {\lambda} +    \frac{\kappa'_1\sigma_m}{y_m} +   \frac{ \kappa'_2} {(ny_m^2)^{1/2}}       x^{1/2} +\left( \frac{9\tau_n}{ny_m^2}   \right) x .
\end{split}
\end{equation*}
It remains to choose $\lambda$ as $\lambda_m = ny_m^2>0$ and $y_m = \theta \left(\frac{x}{n} + \sigma_m^2 + \frac{L\log(K)}{n}\right)^{1/2}$, with $\theta>1$ to be explicited later on. We can already write
\begin{equation*}
\begin{split}
    \frac{1}{n}\sum_{l = 1}^{L}   -\nu_l(\widehat{f}_{\widehat{k}_l})     
\leq \left(y_m^2  +       \kappa'_0 \frac{1}{n}\sum_{l=1}^{L}n_l E   \left[   \mid    \log\left(\frac{\widehat{f}_{\widehat{k}_l}}{s_{l}}\right) \mid  ^2 \right]   \right) \left(    \frac{9\tau_n+1}{\theta ^2} +\frac{\kappa'_1 + \kappa'_2}{\theta} \right)
\end{split}
\end{equation*}
and obtain an upper bound for $  \frac{1}{n}\sum_{l = 1}^{L}   -\nu_l(\widehat{f}_{\widehat{k}_l})  $.
By using equation \eqref{majorationKL}, 
\begin{equation*}
\begin{split}
\frac{1}{n}\sum_{l = 1}^{L} n_l \mathbf{KL}(s_l, \widehat{f}_{\widehat{k_l}})   &\leq    \frac{2 L\log K}{n} +  \inf_{\Kl} \left( \frac{ 1}{n}\sum_{l = 1}^{L} \nu_l(\bar{f}_{k_l}) +   \inf_{f\in \mathcal{F}_m} \left(\sum_{l=1}^{L} \frac{n_l}{n} \mathbf{KL}(s_l, f_{k_l})\right) \right)   + \frac{\delta  + \eta}{n}   \\
&\quad+ \left(y_m^2   +       \kappa'_0\frac{1}{n}\sum_{l=1}^{L}n_l E   \left[   \mid    \log\left(\frac{\widehat{f}_{\widehat{k}_l}}{s_{l}}\right) \mid  ^2 \right]   \right) \left(    \frac{9\tau_n +1}{\theta ^2}+\frac{\kappa'_1 + \kappa'_2}{\theta} \right) .
\end{split}
\end{equation*}
Now we define $C_{\tau_n} = \frac{e^{-\tau_n}+\tau_n -1}{\tau_n^2}$ and choose $\epsilon_{pen}>0$ such that $  \left( \frac{9\tau_n +1}{\theta_{pen} ^2}+\frac{\kappa'_1 + \kappa'_2}{\theta_{pen}} \right)\kappa_0' = C_{\tau_n} \epsilon_{pen}$. Using Proposition \ref{meynet}, we obtain
\begin{equation*}
\begin{split}
\frac{1}{n}\sum_{l = 1}^{L} n_l \mathbf{KL}(s_l, \widehat{f}_{\widehat{k_l}})    &\leq    \frac{ 2L\log K}{n} +   \inf_{\Kl} \left( \frac{1}{n} \sum_{l = 1}^{L} \nu_l(\bar{f}_{k_l}) +  \inf_{f \in \mathcal{F}_m}\left(\sum_{l=1}^{L} \frac{n_l}{n} \mathbf{KL}(s_l, f_{k_l})\right) \right)  \\ &\quad+ \frac{\delta  + \eta}{n}   
+ \frac{\epsilon_{pen}}{n}\sum_{l = 1}^{L}n_l\mathbf{KL}(s_l, \widehat{f}_{\widehat{k_l}}) + \frac{y_m^2 C_{\tau_n}\epsilon_{pen}}{\kappa_0'}.
\end{split}
\end{equation*}
Therefore, with probability less than $e^{-x}$,
\begin{equation*}
\begin{split}
\frac{1-\epsilon_{pen}}{n}\sum_{l = 1}^{L} n_l \mathbf{KL}(s_l, \widehat{f}_{\widehat{k_l}})-  \frac{2 L\log K}{n}        &-  \inf_{\Kl} \left( \frac{1}{n}\sum_{l = 1}^{L} \nu_l(\bar{f}_{k_l})      
+ \inf_{f\in\mathcal{F}_m} \left(\sum_{l=1}^{L} \frac{n_l}{n} \mathbf{KL}(s_l, f_{k_l})\right)  \right)  \\
&- \frac{\delta  + \eta}{n}   
- \frac{\theta_{pen}^2(\sigma_m^2 + \frac{L\log(K)}{n}) C_{\tau_n}\epsilon_{pen}}{\kappa_0'}
> \frac{x}{n}\theta_{pen}^2\frac{C_{\tau_n}\epsilon_{pen}}{\kappa_0'}.
\end{split}
\end{equation*}
For all $\alpha>0$ and any non negative random variable, we have $E[Z] = \alpha\int_{x\geq 0} \text{pr}(Z > \alpha x)dx$. Furthermore, let $\kappa_0 = \frac{C_{\tau_n}\epsilon_{pen}\theta_{pen}^2}{\kappa_0'}$, we get:
\begin{equation*}
\begin{split}
E\left[\frac{1-\epsilon_{pen}}{n}\sum_{l = 1}^{L} n_l \mathbf{KL}(s_l, \widehat{f}_{\widehat{k_l}})     -  \frac{ 2L\log K}{n} -    \inf_{\Kl} \left( \frac{1}{n}\sum_{l = 1}^{L} \nu_l(\bar{f}_{k_l})         
+ \inf_{f\in\mathcal{F}_m} \left(\sum_{l=1}^{L} \frac{n_l}{n} \mathbf{KL}(s_l, f_{k_l})\right)  \right) \right. \\ 
- \left.\frac{\delta  + \eta}{n}  
- \kappa_0 (\sigma_m^2 + \frac{L\log(K)}{n})\right]
\leq \frac{\kappa_0}{n}.
\end{split}
\end{equation*}
Since $E[\frac{1}{n}\sum_{l = 1}^{L} \nu_l(\bar{f}_{k_l})(X_l)] = 0$ for any cluster assignment $\Kl$ of the observations, we derive:
\begin{equation*}
\begin{split}
E\left[\sum_{l = 1}^{L} \frac{n_l}{n} \mathbf{KL}(s_l, \widehat{f}_{\widehat{k_l}}) \right]     &\leq  \frac{1}{1-\epsilon_{pen}}  \left\{  \inf_{\substack{f \in \mathcal{F}_m\\(k_l)_l \in \Iintv{1,K}^L}} \left(\sum_{l=1}^{L} \frac{n_l}{n} \mathbf{KL}(s_l, f_{k_l})\right)  \right.\\
&\quad \left.   +(2+  \kappa_0 )\frac{  L\log K}{n} + \kappa_0 \sigma_m^2   
+  \frac{\kappa_0}{n} + \frac{\delta  + \eta}{n}   \right\}.
\end{split}
\end{equation*}
Recalling that $\delta$ can be chosen arbitrary small, this leads to:
\begin{equation*}
\begin{split}
E\left[\sum_{l = 1}^{L} \frac{n_l}{n}\mathbf{KL}(s_l, \widehat{f}_{\widehat{k_l}}) \right]     &\leq  \frac{1}{1-\epsilon_{pen}}  \left(  \inf_{\substack{f \in \mathcal{F}_m\\(k_l)_l \in \Iintv{1,K}^L}} \left(\sum_{l=1}^{L} \frac{n_l}{n} \mathbf{KL}(s_l, f_{k_l})\right)  \right. \\
&\quad  \left.   +(2+  \kappa_0 )\frac{  L\log K}{n} + \kappa_0 \sigma_m^2 
+  \frac{\kappa_0}{n} + \frac{ \eta}{n}   \right),
\end{split}
\end{equation*}
which concludes the proof by taking $C_1 = \frac{1}{1-\epsilon_{pen}}$ and $C_2 = \frac{\kappa_0}{1-\epsilon_{pen}}$.
\end{proof}

\subsubsection{Proof of Lemma \ref{markovIneq}}
Let $x\geq 0$. Take $A = \{Z < \Psi(x)\}$. Either $\text{pr}(A) = 0 \leq e^{-x}$ or $\text{pr}(A)>0$ in which case, by assumption,
\begin{equation*}
\begin{split}
\Psi(x) <\frac{E[Z\mathbb{1}_{\{Z <\Psi(x)\}}]}{\text{pr}(Z<\Psi(x))}\leq \Psi \left(\log\left(  \frac{1}{\text{pr}(Z<\Psi(x))}  \right)\right).
\end{split}
\end{equation*}
Since $\Psi$ is non-decreasing, $x<\log\left(  \frac{1}{\text{pr}(Z<\Psi(x))}  \right)$, which leads to the conclusion.

\subsubsection{Proof of Lemma \ref{lemmaConc}}
Consider a class of real-valued and measurable functions defined below. For a fixed family of $K$ functions denoted by $\tilde{f} = (\tilde{f}_1,\ldots,\tilde{f}_K)$: 
\begin{equation*}
\begin{split}
\mathcal{G}(\tilde{f}, \sigma) &= \left\{   \log     \left(  \prod_{i = 1}^{n_l}  \frac{f_{k_l}}{\tilde{f}_{k_l}}  \right)_{ l\in\{1,\ldots,L\} }   \mid f \in \mathcal{F}_m , \mathbf{a}(f, \tilde{f}) \leq \sigma ^2 \right\} \\ 
&= \left\{    \left(  -\mathbf{kl}(f_{k_l}) + \mathbf{kl}(\tilde{f}_{k_l})  \right)_{ l\in\{1,\ldots,L\} }     \mid f \in \mathcal{F}_m , \mathbf{a}(f, \tilde{f})  \leq \sigma ^2  \right\}  
\end{split}
\end{equation*} 
We are focusing on $W(\tilde{f}, \sigma) = \sup_{f\in \mathcal{G}(\tilde{f}, \sigma)} \sum_{l=1}^{L} (-\nu_l(f_{k_l}) + \nu_l(\tilde{f}_{k_l})) $. If $[f_{k_l}^-, f_{k_l}^+]$ is a bracket of size $\gamma$ containing $f_{k_l}$, then 
\begin{equation*}
g_{k_l}^- = \log \left( \frac{f_{k_l}^-}{\tilde{f}_{k_l}} \right)   \leq  \log \left( \frac{f_{k_l}}{\tilde{f}_{k_l}} \right)    \leq \log  \left( \frac{f_{k_l}^+}{\tilde{f}_{k_l}} \right)   =   g_{k_l}^+\text{ and }
g_{k_l}^+ -  g_{k_l}^- =  \log  \left( \frac{f_{k_l}^+}{f_{k_l}^-} \right).
\end{equation*} 
According to Assumption \ref{GhypSurModele}, $\mid \log\left( \frac{ f_{k_l}^+}{f_{k_l}^-}\right) \mid \leq \tau_n$. Thus, for any integer $j\geq 2$:
\begin{equation*}
\frac{1}{n}\sum_{l=1}^{L}\sum_{i=1}^{n_l}E_{s_l}\left[   \mid  g_{k_l}^+ -  g_{k_l}^-    \mid ^j \right]  \leq     \frac{j!}{2}  \tau_n^{j-2} \gamma ^2. 
\end{equation*} 
Recall Theorem 6.8 from \cite{Mas2007}:

\begin{theorem}
\label{massartThm}
Let $\mathcal{G}$ be a countable class of real-valued and measurable functions. Assume that there exist some positive numbers $V$ and $b$ such that for all $f \in \mathcal{G}$ and all integers $j\geq 2$,
\begin{equation*}
E[\mid f \mid^j ] \leq \frac{j!}{2}Vb^{j-2}.
\end{equation*}
Assume furthermore that for any positive number $\gamma$, there exist a finite set $\mathcal{B}(\gamma)$ of brackets covering $\mathcal{G}$ such that for any bracket $[g^-, g^+] \in \mathcal{B}(\gamma)$ and all integers $k\geq 2$,
\begin{equation*}
E[\mid  g^+ - g^-  \mid^j] \leq \frac{j!}{2} \gamma^2 b^{j-2}.
\end{equation*}
Let $e^{H(\gamma)}$ denote the minimal cardinality of such a covering. Then, there exists an absolute constant $\kappa$ such that for any $ \epsilon \in (0,1]$ and any measurable set $A$ with $\text{pr}(A)>0$,
\begin{equation*}
E^A\left[   \frac{1}{n}    \sup_{f\in\mathcal{G}}  \sum_{l=1}^{L} \nu_l(f) \right] \leq E + \frac{(1+6\epsilon)(2V)^{1/2}}{n^{1/2}}\log\left(\frac{1}{\text{pr}(A)}\right)^{1/2} + \frac{2b}{n}\log\left( \frac{1}{\text{pr}(A)}  \right),
\end{equation*}
where $E = \frac{\kappa}{\epsilon} \frac{1}{\sqrt{n}} \int_{0}^{\epsilon V^{1/2}}(H(\gamma) \land n )^{1/2}d\gamma +\frac{2(b+\sqrt{V})}{n} H(V^{1/2}) $. Furthermore, $\kappa \leq 27$.
\end{theorem}

In our context, the assumptions of the theorem are satisfied on $\mathcal{G}(\tilde{f}, \sigma)$ with $V =  \sigma^2$ and $b= \tau_n$ except that the set $\mathcal{G}(\tilde{f}, \sigma)$ is not necessarily countable so the supremum $W(\tilde{f}, \sigma)$ can be non measurable. We rather define the countable subset: 
\begin{equation*}
\begin{split}
\mathcal{G}'(\tilde{f}, \sigma) = \left\{    \left(  -\mathbf{kl}(f_{k_l}) + \mathbf{kl}(\tilde{f}_{k_l})  \right)_{ l\in\{1,\ldots,L\} }     \mid f \in \mathcal{F}'_m , \mathbf{a}(f, \tilde{f})  \leq \sigma ^2  \right\},  
\end{split}
\end{equation*} 
with $\mathcal{F}'_m$ satisfying Assumption \ref{hypSep}. Thus, $W(\tilde{f}, \sigma) = \sup_{f\in \mathcal{G}'(\tilde{f}, \sigma)} \sum_{l=1}^{L} (-\nu_l(f_{k_l}) + \nu_l(\tilde{f}_{k_l}))$ almost surely, and by applying Theorem \ref{massartThm}, we can conclude that
\begin{equation*}
\begin{split}
E^A\left[  \sup_{f\in\mathcal{G}(\tilde{f}, \sigma)} \frac{1}{n} \sum_{l=1}^{L}(-\nu_l(f_{k_l}) + \nu_l(\tilde{f}_{k_l}))  \right] &\leq E + \frac{(1 + 6\epsilon)\sigma 2^{1/2}}{n^{1/2}}\log\left(\frac{1}{\text{pr}(A)}\right)^{1/2}
+ \frac{2\tau_n}{n}\log\left( \frac{1}{\text{pr}(A)}  \right).
\end{split}
\end{equation*}
Let find an upper bound for $E$. Take $\epsilon = 1$. We have:
\begin{equation*}
\begin{split}
E = \frac{ \kappa}{n^{1/2}} \int_{0}^{\sigma}(H(\gamma) \land n )^{1/2}d\gamma +\frac{2(\tau_n+\sigma)}{n} H(\sigma).
\end{split}
\end{equation*}
The mapping $\gamma \mapsto H(\gamma, \mathcal{F}_m(\tilde{f}, \sigma))$ is non-increasing. By Assumption \ref{hypPhi}, if $\tilde{f}\in \mathcal{F}_m$, 
\begin{equation*}
\begin{split}
\int_{0}^{\sigma}(H(\gamma, \mathcal{F}_m(\tilde{f}, \sigma))\land n)^{1/2}d\gamma \leq \phi_m(\sigma).
\end{split}
\end{equation*}
Also,
\begin{equation*}
\begin{split}
H(\sigma, \mathcal{F}_m(\tilde{f}, \sigma)) = \frac{1}{\sigma}\int_{0}^{\sigma} H(\sigma, \mathcal{F}_m(\tilde{f}, \sigma)) d\gamma \leq      \left(\frac{1}{\sigma}\int_{0}^{\sigma} H(\gamma, \mathcal{F}_m(\tilde{f}, \sigma)) d\gamma  \right)^2   \leq \frac{\phi_m^2(\sigma)}{\sigma^2}.
\end{split}
\end{equation*}
By inserting these bounds,
\begin{equation*}
\begin{split}
E \leq \kappa \frac{1}{n^{1/2}} \phi_m(\sigma) +\frac{2(\tau_n+\sigma)}{n} \frac{\phi_m^2(\sigma)}{\sigma^2}
\leq  \left(\kappa   +2(\tau_n+\sigma)\frac{\phi_m(\sigma)}{n^{1/2}\sigma^2}\right) \frac{\phi_m(\sigma)}{n^{1/2}}.
\end{split}
\end{equation*}
Since $\delta \mapsto \delta^{-1}\phi_m(\delta)$ is non-increasing, so is $\delta \mapsto \delta^{-2}\phi_m(\delta)$. Also, by definition of $\sigma_m$, $\frac{\phi_m(\sigma_m)}{n^{1/2}\sigma_m^2} = 1$. Thus, when $\sigma\geq \sigma_m$,
\begin{equation*}
\begin{split}
E \leq \left(\kappa   +2(\tau_n+\sigma)\right) \frac{\phi_m(\sigma)}{n^{1/2}} \leq \left(27   +2(b+\sigma)\right) \frac{\phi_m(\sigma)}{{n}^{1/2}}.
\end{split}
\end{equation*}
Notice also that for any family $f,g\in \mathcal{F}_m$:
\begin{equation*}
\begin{split}
\mathbf{a}(f,g) = \frac{1}{n}\sum_{l=1}^{L}n_l   E_{s_l}   \left[   \mid    \log\left(\frac{f_{k_l}}{g_{k_l}}\right) \mid  ^2 \right]     
&\leq      \frac{1}{n}\sum_{l=1}^{L}n_l  \tau_n^2 = \tau_n^2.
\end{split}
\end{equation*}
Therefore, $\sigma \leq \tau_n$ and for all $\sigma \geq \sigma_m$:
\begin{equation*}
\begin{split}
E^A \left[   \sup_{f \in \mathcal{G}(\tilde{f}, \sigma)}  \frac{1}{n} \sum_{l=1}^{L}(-\nu_l(f_{k_l}) + \nu_l(\tilde{f}_{k_l}))   \right] &\leq (27 +4\tau_n )\frac{\phi_m(\sigma)}{n^{1/2}}+ \frac{7\times2^{1/2}}{n^{1/2}}\sigma\log \left( \frac{1}{\Proba(A)}  \right)^{1/2} 
+ \frac{2\tau_n}{n} \log \left( \frac{1}{\Proba(A)} \right).
\end{split}
\end{equation*}
We now use the pealing lemma as stated in Lemma 4.23 from \cite{Mas2007} in order to bound the supremum on the overall model. 

\begin{lemma}[\textbf{Pealing lemma}]
\label{pealing}
Let $S$ be a countable set, $\tilde{f} \in S$ and $a : S \rightarrow \mathbb{R}^+$ such that $a(\tilde{f}) = \inf_{f\in S}a(f)$. 
Let $Z$ be a random process indexed by $S$ and $B(\sigma) = \{ f\in S \mid a(f) \leq \sigma\}$. 
Assume that for any positive $\sigma$ the non-negative random-variable $\sup_{f \in B(\sigma)} (Z(f) - Z(\tilde{f}))$ has finite expectation. Then, for any function $\psi$ on $\mathbb{R}^+$ such that $\frac{\psi(x)}{x}$ is non-increasing on $\mathbb{R}^+$ and
\begin{equation*}
E\left[    \sup_{f\in B(\sigma)} Z(f) - Z(\tilde{f})   \right] \leq \psi(\sigma) , \sigma \geq \sigma_{\star}\geq 0,
\end{equation*}
one has for any positive $x\geq \sigma_{\star}$:
\begin{equation*}
E\left[   \sup_{f\in S}  \frac{Z(f) - Z(\tilde{f})}{x^2 + a^2(f)} \right]  \leq 4 \frac{\psi(x)}{x^2}.
\end{equation*}
\end{lemma}

With $S = \mathcal{F}_m$, $\sigma_{\star} = \sigma_m$, $\tilde{f}$ to be specified later, $a(f) = \frac{1}{n}\sum_{l=1}^{L}n_lE   \left[   \mid     \log\left(\frac{f_{k_l}}{\tilde{f}_{k_l}}\right) \mid  ^2 \right] $, $Z(f) = \frac{1}{n}\sum_{l = 1}^{L} -\nu_l(f_{k_l})$ and $Z(\tilde{f})=  \frac{1}{n}\sum_{l = 1}^{L} -\nu_l(\tilde{f}_{k_l})$, provided $y_m \geq \sigma_m$, we have:
\begin{equation*}
\begin{split}
E^A \left[ \sup_{f \in \mathcal{F}_m}  \frac{1}{n}\sum_{l = 1}^{L}    \frac{-\nu_l(f_{k_l}) + \nu_l(\tilde{f}_{k_l})}{y_m^2  + \frac{1}{n}\sum_{l=1}^{L}n_lE _{s_l}  \left[   \mid    \log\left(\frac{f_{k_l}}{\tilde{f}_{k_l}}\right) \mid  ^2 \right]   }\right]
&\leq
4(27 +4\tau_n)\frac{\phi(y_m)}{n^{1/2}y_m^2}+ \frac{28\times 2^{1/2}}{n^{1/2}y_m}\log \left( \frac{1}{\Proba(A)}  \right)^{1/2} \\
&\quad+ \frac{8\tau_n}{ny_m^2} \log \left( \frac{1}{\Proba(A)} \right).
\end{split}
\end{equation*}
Using the monotonicity of $\delta \mapsto \phi(\delta)/\delta$ and the definition of $\sigma_m$,
\begin{equation*}
\begin{split}
E^A \left[ \sup_{f \in \mathcal{F}_m}  \frac{1}{n}\sum_{l = 1}^{L}    \frac{-\nu_l(f_{k_l}) + \nu_l(\tilde{f}_{k_l})}{y_m^2  +\frac{1}{n}\sum_{l=1}^{L}n_lE _{s_l}  \left[   \mid    \log\left(\frac{f_{k_l}}{\tilde{f}_{k_l}}\right) \mid  ^2 \right]  }\right]
&\leq
4(27 +4\tau_n)\frac{\sigma_m}{y_m} + \frac{28\sqrt{2}}{n^{1/2}y_m}\log \left( \frac{1}{\Proba(A)}  \right)^{1/2} \\
&\quad+ \frac{8\tau_n}{ny_m^2} \log \left( \frac{1}{\Proba(A)} \right).
\end{split}
\end{equation*}
We have chosen $\tilde{f}$ such that for any $f$ family of $K$ functions and any $\epsilon_d>0$,
\begin{equation*}
\frac{1}{n}\sum_{l=1}^{L}n_lE   \left[   \mid    \log\left(\frac{s_l}{\tilde{f}_{k_l}}\right) \mid  ^2 \right] \leq (1+\epsilon_d) \frac{1}{n}\sum_{l=1}^{L}n_lE   \left[   \mid    \log\left(\frac{s_l}{f_{k_l}}\right) \mid  ^2 \right].
\end{equation*}
Therefore, $\frac{1}{n}\sum_{l=1}^{L}n_lE   \left[   \mid    \log\left(\frac{f_{k_l}}{\tilde{f}_{k_l}}\right) \mid  ^2 \right]   \leq 2(2+\epsilon_d) \frac{1}{n}\sum_{l=1}^{L}n_l E   \left[   \mid    \log\left(\frac{f_{k_l}}{s_{l}}\right) \mid  ^2 \right]  $ and
\begin{equation*}
\begin{split}
E^A \left[ \sup_{f \in \mathcal{F}_m}  \frac{1}{n}\sum_{l = 1}^{L}    \frac{-\nu_l(f_{k_l}) + \nu_l(\tilde{f}_{k_l})}{y_m^2  + 2(2+\epsilon_d)\frac{1}{n}\sum_{l=1}^{L}n_lE   \left[   \mid    \log\left(\frac{f_{k_l}}{s_{l}}\right) \mid  ^2 \right] }\right]
&\leq
4(27 +4\tau_n)\frac{\sigma_m}{y_m} \\
&\quad + \frac{28\times 2^{1/2}}{n^{1/2}y_m}\log \left( \frac{1}{\Proba(A)}  \right)^{1/2}\\
&\quad + \frac{8\tau_n}{ny_m^2} \log \left( \frac{1}{\Proba(A)} \right).
\end{split}
\end{equation*}
We now use a Bernstein-type control, which is a rewriting of Bernstein's theorem:

\begin{lemma}[\textbf{Bernstein Inequality}]
\label{bernstein}
Assume there exist $V',b'\geq 0$ such that for all integer $j\geq 2$ 
\begin{equation*}
\frac{1}{n}\sum_{l = 1}^{L} n_lE\left[\log\left(\frac{f_{k_l}}{s_l}\right)_+^j\right] \leq \frac{j!}{2}V'b'^{j-2}.
\end{equation*}
Then, for all measurable event $A$ such that $\text{pr}(A)>0$,
\begin{equation*}
E^A\left[\frac{1}{n}\sum_{l = 1}^{L}-\nu_l(f_{k_l})\right] \leq \left(\frac{2V'}{n}\log(\frac{1}{\text{pr}(A)})\right)^{1/2} + \frac{b'}{n}\log \left(  \frac{1}{\text{pr}(A)} \right).
\end{equation*}
\end{lemma}

By taking $V'= \frac{1}{n}\sum_{l=1}^{L}n_lE   \left[   \mid    \log\left(\frac{\tilde{f}_{k_l}}{s_{l}}\right) \mid  ^2 \right]$ and $b' = \tau_n$ and applying this result to $\tilde{f}$, this yields for all $y_m, \kappa'>0$:
\begin{equation*}
\begin{split}
E^A \left[  \frac{1}{n}\sum_{l = 1}^{L}   \frac{- \nu_l(\tilde{f}_{k_l}) }{y_m^2 +   \kappa'^2 \frac{1}{n}  \sum_{l=1}^{L}n_l  E   \left[   \mid    \log\left(\frac{s_l}{\tilde{f}_{k_l}}\right) \mid  ^2 \right]  } \right ]  
&\leq   \frac{1}{{y_m^2 +   \kappa'^2 \frac{1}{n} \sum_{l=1}^{L}n_lE   \left[   \mid    \log\left(\frac{s_l}{\tilde{f}_{k_l}}\right) \mid  ^2 \right]  }}    \left(\frac{  2V'  }{n}\right)^{1/2}    \log    \left(\frac{1}{\Proba(A)}     \right)^{1/2}      \\
&\quad+    \frac{1}{{y_m^2 +   \kappa '^2 \frac{1}{n} \sum_{l=1}^{L}n_lE   \left[   \mid    \log\left(\frac{s_l}{\tilde{f}_{k_l}}\right) \mid  ^2 \right] }}  \frac{b'}{n}\log\left(\frac{1}{\Proba(A)}\right) \\
&\leq      \frac{1}{\kappa'}   \left( \frac{2} {ny_m^2} \right)^{1/2}  \log    \left(\frac{1}{\Proba(A)}     \right)^{1/2}     +     \frac{1}{y_m^2 } \frac{b'}{n}\log\left(\frac{1}{\Proba(A)}\right).
\end{split}
\end{equation*}
Therefore, 
\begin{equation*}
\begin{split}
E^A \left[  \sup_{f \in \mathcal{F}_m}  \frac{1}{n}\sum_{l = 1}^{L}    \frac{-\nu_l(f_{k_l}) + \nu_l(\tilde{f}_{k_l})}{y_m^2  +      2(2+\epsilon_d)\frac{1}{n}  \sum_{l=1}^{L}n_lE   \left[   \mid    \log\left(\frac{f_{k_l}}{s_{l}}\right) \mid  ^2 \right]  }  \right.
 \left.+  \frac{1}{n}\sum_{l = 1}^{L}   \frac{- \nu_l(\tilde{f}_{k_l}) }{y_m^2 +   \kappa'^2 \frac{1}{n}  \sum_{l=1}^{L}n_l   E   \left[   \mid    \log\left(\frac{s_l}{\tilde{f}_{k_l}}\right) \mid  ^2 \right]   }  \right] \\  
\leq 4(27 +4\tau_n)\frac{\sigma_m}{y_m} +  \left(  \frac{28\times 2^{1/2}} {(ny_m^2)^{1/2}}  +   \frac{1}{\kappa'}    \left(\frac{2} {ny_m^2} \right)^{1/2} \right)      \log \left( \frac{1}{\Proba(A)}  \right)^{1/2} 
+\left( \frac{9\tau_n}{ny_m^2}   \right) \log \left( \frac{1}{\Proba(A)} \right).
\end{split}
\end{equation*}
Choosing $\kappa'$ as $\kappa'^2_d = \frac{2(2+\epsilon_d)}{1+\epsilon_d}$, we can conclude that since:
\begin{equation*}
\begin{split}
E^A \left[  \sup_{f \in \mathcal{F}_m}  \frac{1}{n}\sum_{l = 1}^{L}    \frac{-\nu_l(f_{k_l}) + \nu_l(\tilde{f}_{k_l})}{y_m^2  +      2(2+\epsilon_d) \frac{1}{n}\sum_{l=1}^{L}n_l E   \left[   \mid    \log\left(\frac{f_{k_l}}{s_{l}}\right) \mid  ^2 \right]  }  +  \right.
\left.\frac{1}{n}\sum_{l = 1}^{L}   \frac{- \nu_l(\tilde{f}_{k_l}) }{y_m^2 +   \kappa'^2   \frac{1}{n} \sum_{l=1}^{L}n_l  E   \left[   \mid    \log\left(\frac{s_l}{\tilde{f}_{k_l}}\right) \mid  ^2 \right]   }  \right] \\  
\geq E^A \left[  \sup_{f \in \mathcal{F}_m}  \frac{1}{n}\sum_{l = 1}^{L}    \frac{-\nu_l(f_{k_l}) + \nu_l(\tilde{f}_{k_l})}{y_m^2  +      2(2+\epsilon_d) \frac{1}{n}\sum_{l=1}^{L}n_l E   \left[   \mid    \log\left(\frac{f_{k_l}}{s_{l}}\right) \mid  ^2 \right]  }  \right.
\left.+  \frac{1}{n}\sum_{l = 1}^{L}   \frac{- \nu_l(\tilde{f}_{k_l}) }{  y_m^2  +      2(2+\epsilon_d)\frac{1}{n} \sum_{l=1}^{L}n_l E   \left[   \mid    \log\left(\frac{f_{k_l}}{s_{l}}\right) \mid  ^2 \right] }  \right] ,
\end{split}
\end{equation*}
we have
\begin{equation*}
\begin{split}
E^A \left[  \sup_{f \in \mathcal{F}_m}  \frac{1}{n}\sum_{l = 1}^{L}    \frac{-\nu_l(f_{k_l}) }{y_m^2  +      2(2+\epsilon_d) \frac{1}{n}\sum_{l=1}^{L}n_l  E   \left[   \mid    \log\left(\frac{f_{k_l}}{s_{l}}\right) \mid  ^2 \right]  }  \right ]  \\
\leq 4(27 + 4\tau_n )\frac{\sigma_m}{y_m}
+  \left(  \frac{28\times 2^{1/2}} {(ny_m^2)^{1/2}}  +   \frac{1}{\kappa'_d}   \left( \frac{2} {ny_m^2} \right)^{1/2}  \right)       \log \left( \frac{1}{\Proba(A)}  \right)^{1/2} 
+ \frac{9\tau_n}{ny_m^2}   \log \left( \frac{1}{\Proba(A)} \right).
\end{split}
\end{equation*}
Defining $\kappa'_1 = 4(27 + 4\tau_n)$ , $\kappa'_2 = 2^{1/2}(28 + \frac{1}{\kappa'_d})$ and $\kappa'_0 = 2(2+\epsilon_d)$ leads to the conclusion.

\subsection{Model selection theorem}
Just as model complexity appeared in the single model inequality, the multi-model case involves a term that takes the global collection into account. Therefore, we assume the existence of the following Kraft inequality which bounds in a sense the complexity of our collection of models:

\begin{assumption}[\textbf{Kraft inequality}]
\label{hypKraft}
There exists a family $(x_m)_{m \in \mathcal{M}}$ of non-negative numbers such that
\begin{equation*}
\sum_{m\in \mathcal{M}} e^{-x_m} \leq \Sigma < + \infty.
\end{equation*}
\end{assumption}

\begin{theorem}
\label{thmModelSelection}
Let $(x_1,\ldots,x_L)$ be $L$ observations of independent random vectors $(X_1,\ldots,X_L)$ where each $X_l$ consists of $n_l$ independent and identically distributed instances of a multinomial vector that has $s_l$ as a true categorical density with respect to some known positive measure. Assume $(\mathcal{S}_m)_{m\in\mathcal{M}}$ is an at most countable collection of models for which Assumption \ref{hypKraft} holds. For every model $S_m\in \mathcal{S}$, we also assume that Assumptions \ref{GhypSurModele}, \ref{hypPhi} and \ref{hypSep} hold. Let $\mathbf{pen}$ be a non-negative penalty function and $\widehat{m}$ any $\eta'$-minimizer of
\begin{equation*}
\mathbf{crit}(m) =   \min_{(\pi^m ,f^m )\in \Theta_{m}}    \gamma_n(\pi^m, f^m)   +\mathbf{pen}(m).
\end{equation*}
Let $(\widehat{\pi}^{\widehat{m}}, \widehat{f}^{\widehat{m}})$ be the corresponding $\eta$-likelihood minimizers in $\mathcal{S}_{\widehat{m}}$ and define the resulting cluster assignment for each vector  $x_l = (x_l^1,\ldots,x_l^{n_l})$:
\begin{equation*}
\widehat{k_l} = \argmax_{k \in \Iintv{1,\widehat{K}}} \left\{ \widehat{\pi}_k^{\widehat{m}} \left(\prod_{i = 1}^{n_l} \widehat{f}_k^{\widehat{m}}(x_l^i) \right) \right\}.
\end{equation*}
Define $\mathfrak{D}_m = n\sigma_m^2$ with $\sigma_m$ the unique root of $\frac{1}{\sigma}\phi_m({\sigma}) = \sqrt{n}\sigma$. Then, for any $C_1>0$, there exist some constants $\kappa_0$ and $C_2$ that depend only on $C_1$, and such that whenever
\begin{equation*}
\mathbf{pen}(m)\geq \kappa (\mathfrak{D}_m + L \log(K) +  x_m) \text{ with } \kappa > 1 + \kappa_0,
\end{equation*}
for all model $m\in\mathcal{M}$, the penalized log-likelihood estimate satisfies 
\begin{equation*}
\begin{split}
E\left[\sum_{l =1}^{L}  \frac{n_l}{n} \mathbf{KL}(s_l, \widehat{f}_{\widehat{k_l}}^{\widehat{m}} ) \right]
&\leq C_1   \inf_{m\in \mathcal{M}} \left\{
 \inf_{\substack{f \in \mathcal{F}_m\\(k_l)_l \in \Iintv{1,K}^L}} \left(\sum_{l = 1}^{L}\frac{n_l}{n} \mathbf{KL}(s_l, f_{k_l}^m) + \frac{\mathbf{pen}(m)}{n}\right)   \right\}\\
&\quad+ C_2 \frac{\Sigma}{n} + \frac{\eta  + \eta' }{n}.
\end{split}
\end{equation*}
\end{theorem}


\begin{proof}
For any cluster assignment $(k_1,\ldots,k_L)$ of the observations within the model $m$, define $\overbar{f}^m \in \mathcal{F}_m$ such that:
\begin{equation}
\label{klbar}
\sum_{l = 1}^{L} n_l \mathbf{KL}(s_l, \overbar{f}_{k_l}^m) \leq \inf_{f^m\in\mathcal{F}_m} \sum_{l=1}^{L} n_l \mathbf{KL}(s_l, f_{k_l}^m) +\delta.
\end{equation}
Fix also $m \in \mathcal{M}$ such that $\sum_{l = 1}^{L} n_l \mathbf{KL}(s_l, \overbar{f}_{k_l}^m) < +\infty$ and define
\begin{equation*}
\begin{split}
\mathcal{M}' = \left\{m' \in \mathcal{M} \mid \gamma_n(\widehat{\pi}^{m'}, \widehat{f}^{m'})    + \mathbf{pen}(m') \right.
\leq \left.\gamma_n(\widehat{\pi}^m, \widehat{f}^m)    + \kappa_0(\mathfrak{D}_m + L\log(K) + x_m) + \eta'\right \}.
\end{split}
\end{equation*}
By definition of the estimator and since $-\log$ is decreasing,
\begin{flalign*}
\begin{split}
\gamma_n(\widehat{\pi}^m, \widehat{f}^m)  + \kappa_0(\mathfrak{D}_m + L\log(K) + x_m) + \eta' &\leq  \inf_{(\pi^m ,f^m )\in \Theta_{m}}     \gamma_n(\pi^m, f^m)  + \eta \\
&\quad+  \kappa_0(\mathfrak{D}_m + L\log(K) + x_m) + \eta'\\
&\leq \inf_{\Kl}\left\{\inf_{(\pi^m ,f^m )\in \Theta_{m}} \left(\sum_{l=1}^{L} -\log(\pi^m_{k_l}) \right.\right.\\
&\quad + \left.\left. \sum_{l=1}^{L}\sum_{i = 1}^{n_l} - \log   \left( \left( \frac{f_{k_l}^{m}}{s_l}  \right)  (X_l^i) \right) \right) \right\}\\
&\quad+ \eta+  \kappa_0(\mathfrak{D}_m + L\log(K) + x_m) + \eta'.
\end{split}
\end{flalign*}
\normalsize
Using the previous definition of $\nu_l$ in (\ref{defCentralized}) besides equation (\ref{klbar}) and by assumption on \textbf{pen}, 
\begin{flalign*}
\begin{split}
&\inf_{\Kl}\left(\inf_{(\pi^m ,f^m )\in \Theta_{m}} \left(\sum_{l=1}^{L} -\log(\pi^m_{k_l}) +  \sum_{l=1}^{L}\sum_{i = 1}^{n_l} - \log   \left( \left( \frac{f_{k_l}^{m}}{s_l}  \right)  (x_l^i) \right) \right)  \right)\\
&\quad+ \eta+  \kappa_0(\mathfrak{D}_m + L\log(K) + x_m) + \eta'\\
&\leq \inf_{\Kl}\left( \sum_{l=1}^{L} -\log(\frac{1}{K})  +  \sum_{l=1}^{L}\nu_l(\bar{f}^m_{k_l})  + \sum_{l = 1}^{L}n_l \mathbf{KL}(s_l, \bar{f}^m_{k_l}) \right)\\
&\quad+\eta+  \kappa_0(\mathfrak{D}_m + L\log(K) + x_m) + \eta'\\
&\leq L\log(K) + \inf_{\Kl}\left(  \sum_{l=1}^{L}\nu_l(\bar{f}^m_{k_l}) +  \inf_{f^m\in\mathcal{F}_m}\sum_{l = 1}^{L}n_l \mathbf{KL}(s_l, f_{k_l}^m) \right) + \delta\\
&\quad + \eta  + \eta' +  \kappa_0(\mathfrak{D}_m + L\log(K) + x_m)\\
&\leq   \inf_{\Kl}\left(  \sum_{l=1}^{L}\nu_l(\bar{f}^m_{k_l}) +  \inf_{f^m\in\mathcal{F}_m}\sum_{l = 1}^{L}n_l \mathbf{KL}(s_l, f_{k_l}^m) \right) + \delta + \eta  + \eta' + \mathbf{pen}(m).
\end{split}
\end{flalign*}
\normalsize
It follows that
\begin{equation*}
\begin{split}
\gamma_n(\widehat{\pi}^{m'}, \widehat{f}^{m'})      + \mathbf{pen}(m') 
&\leq   \inf_{\Kl}\left( \sum_{l=1}^{L}\nu_l(\bar{f}^m_{k_l}) +  \inf_{f^m\in\mathcal{F}_m}\sum_{l = 1}^{L}n_l \mathbf{KL}(s_l, f_{k_l}^m) \right) + \delta \\
&\quad+ \eta  + \eta' + \mathbf{pen}(m)
\end{split}
\end{equation*}
\normalsize
and on the other hand, 
\begin{equation*}
\begin{split}
\gamma_n(\widehat{\pi}^{m'}, \widehat{f}^{m'}) + \mathbf{pen}(m')&\geq 
\sum_{l = 1}^{L}-\log     \left(   K'\widehat{\pi}_{\widehat{k_l}}^{m'} \left(\prod_{i = 1}^{n_l} \left( \frac{\widehat{f}_{k_l}^{m'}}{s_l}\right)  (x_l^i) \right)   \right)   + \mathbf{pen}(m')\\
&\geq \sum_{l = 1}^{L}-\log     \left(   K'\left(\prod_{i = 1}^{n_l} \left( \frac{\widehat{f}_{k_l}^{m'}}{s_l}\right)  (x_l^i) \right)   \right)   + \mathbf{pen}(m')\\
&\geq -L\log( K') + \sum_{l = 1}^{L}  \nu_l(\widehat{f}_{\widehat{k_l}}^{m'})      +  \sum_{l = 1}^{L}  n_l \mathbf{KL}(s_l, \widehat{f}_{\widehat{k_l}}^{m'})+ \mathbf{pen}(m'),
\end{split}
\end{equation*}
where $K'$ denotes the number of clusters given in model $m'$. Thus, 
\begin{equation*}
\begin{split}
\sum_{l =1}^{L}  n_l \mathbf{KL}(s_l, \widehat{f}_{\widehat{k_l}}^{m'}) &\leq L\log( K') - \sum_{l = 1}^{L}  \nu_l(\widehat{f}_{\widehat{k_l}}^{m'})      - \mathbf{pen}(m')\\
&\quad+\inf_{\Kl}\left(  \sum_{l=1}^{L}\nu_l(\bar{f}^m_{k_l}) +  \inf_{f^m\in\mathcal{F}_m}\sum_{l = 1}^{L}n_l \mathbf{KL}(s_l, f_{k_l}^m) \right) + \delta \\
&\quad+ \eta  + \eta' + \mathbf{pen}(m)  
\end{split}
\end{equation*}
\normalsize
It remains to get an upper bound of the deviation $-\sum_{l = 1}^{L}\nu_l(\widehat{f_{\widehat{k_l}}^{m'}})$. Using Lemmas \ref{markovIneq} and \ref{lemmaConc}, except on a set of probability less than $e^{-x_m'-x}$, for any $y_{m'}> \sigma_{m'}$, 
\begin{equation*}
\begin{split}
\frac{1}{n}\sum_{l = 1}^{L}    \frac{-\nu_l(\widehat{f}_{\widehat{k}_l}^{m'})}{y_{m'}^2  +      \kappa'_0\frac{1}{n}\sum_{l=1}^{L}n_l E   \left[   \mid    \log\left(\frac{\widehat{f}_{\widehat{k}_l}^{m'}}{s_{l}}\right) \mid  ^2 \right]  }    
&\leq  \frac{ L \log(K') } {\lambda} +    \frac{\kappa'_1\sigma_{m'}}{y_{m'}} +   \frac{ \kappa'_2} {(ny_{m'}^2)^{1/2}}       (x + x_{m'})^{1/2}\\
&\quad +\left( \frac{9\tau_n}{ny_{m'}^2}   \right)( x + x_{m'}) .
\end{split}
\end{equation*}
This time we choose $\lambda$ as $\lambda_m = ny_{m'}^2>0$ and $y_{m'} = \theta (\frac{x+x_{m'}}{n} + \sigma_{m'}^2 + \frac{L\log(K')}{n})^{1/2}$, with $\theta>1$ to be explicited later on. We deduce that except on a set of probability less than $e^{-x_m'-x}$, for any $y_{m'}> \sigma_{m'}$,
\begin{equation*}
\begin{split}
    \frac{1}{n}\sum_{l = 1}^{L}    \frac{-\nu_l(\widehat{f}_{\widehat{k}_l}^{m'})}{y_{m'}^2  +      \kappa'_0\frac{1}{n}\sum_{l=1}^{L}n_l E   \left[   \mid    \log\left(\frac{\widehat{f}_{\widehat{k}_l}^{m'}}{s_{l}}\right) \mid  ^2 \right]  }   
\leq  \left(    \frac{9\tau_n+1}{\theta ^2} +\frac{\kappa'_1 + \kappa'_2}{\theta} \right).
\end{split}
\end{equation*}
Now, we use Kraft condition \ref{hypKraft}, and conclude that if we make a proper choice of $y_{m'}$ for all models $m'\in \mathcal{M}'$, this property holds simultaneously on $\mathcal{M}'$ except on a set of probability less than $e^{-x}\Sigma$. Therefore, except on that set, for all $m' \in \mathcal{M}'$, 
\begin{equation*}
\begin{split}
\frac{1}{n}\sum_{l =1}^{L}  n_l \mathbf{KL}(s_l, \widehat{f}_{\widehat{k_l}}^{m'} ) &\leq  \frac{ L\log( K') }{n} +\left(y_{m'}^2  +      \kappa'_0\frac{1}{n}\sum_{l=1}^{L}n_l E   \left[   \mid    \log\left(\frac{\widehat{f}_{\widehat{k}_l}^{m'}}{s_{l}}\right) \mid  ^2 \right] \right)  \left(    \frac{9\tau_n+1}{\theta ^2} +\frac{\kappa'_1 + \kappa'_2}{\theta} \right) \\
&\quad - \frac{\mathbf{pen}(m')}{n}  +\inf_{\Kl}\left( \frac{1}{n}   \sum_{l=1}^{L}\nu_l(\bar{f}^m_{k_l})  +\inf_{f^m\in\mathcal{F}_m}\sum_{l = 1}^{L}n_l \mathbf{KL}(s_l, f_{k_l}^m)\right) \\
&\quad+ \frac{\delta}{n}  + \frac{\eta  + \eta' }{n} + \frac{\mathbf{pen}(m)}{n}.
\end{split}
\end{equation*}
\normalsize
Define $C_{\tau_n} = \frac{e^{-\tau_n}+\tau_n -1}{\tau_n^2}$ and choose $\epsilon_{pen}>0$ such that $  \left( \frac{9\tau_n +1}{\theta_{pen} ^2}+\frac{\kappa'_1 + \kappa'_2}{\theta_{pen}} \right)\kappa_0' = C_{\tau_n} \epsilon_{pen}$. We obtain
\begin{equation*}
\begin{split}
\frac{1}{n}\sum_{l =1}^{L}  n_l \mathbf{KL}(s_l, \widehat{f}_{\widehat{k_l}}^{m'} ) &\leq  \frac{ L\log( K') }{n} +\frac{y_{m'}^2 C_{\tau_n}\epsilon_{pen}}{\kappa_0'} + \frac{C_{\tau_n}\epsilon_{pen}}{n}\sum_{l=1}^{L}n_l E   \left[   \mid    \log\left(\frac{\widehat{f}_{\widehat{k}_l}^{m'}}{s_{l}}\right) \mid  ^2 \right] \\
&\quad - \frac{\mathbf{pen}(m')}{n}  +\inf_{\Kl}\left( \frac{1}{n}   \sum_{l=1}^{L}\nu_l(\bar{f}^m_{k_l})  +\inf_{f^m\in\mathcal{F}_m}\sum_{l = 1}^{L}n_l \mathbf{KL}(s_l, f_{k_l}^m)\right)\\
&\quad+ \frac{\delta}{n}  + \frac{\eta  + \eta' }{n} + \frac{\mathbf{pen}(m)}{n}.
\end{split}
\end{equation*}
\normalsize
Then, by using Proposition \ref{meynet}, simultaneously for any $m'\in \mathcal{M}'$ and except on a set of probability less than $e^{-x}\Sigma$, 
\begin{equation}
\label{eqPen}
\begin{split}
(1- \epsilon_{pen})\frac{1}{n}\sum_{l =1}^{L}  n_l \mathbf{KL}(s_l, \widehat{f}_{\widehat{k_l}}^{m'}) &\leq \frac{ L\log( K') }{n} +\frac{y_{m'}^2 C_{\tau_n}\epsilon_{pen}}{\kappa_0'} - \frac{\mathbf{pen}(m')}{n}   \\
&\quad +\inf_{\Kl}\left( \frac{1}{n}   \sum_{l=1}^{L}\nu_l(\bar{f}^m_{k_l})  +\inf_{f^m\in\mathcal{F}_m}\sum_{l = 1}^{L}n_l \mathbf{KL}(s_l, f_{k_l}^m)\right)\\
&\quad+ \frac{\delta}{n}  + \frac{\eta  + \eta' }{n} + \frac{\mathbf{pen}(m)}{n}.
\end{split}
\end{equation}
Let now study the term $ \frac{ L\log( K') }{n} +  \frac{y_{m'}^2 C_{\tau_n}\epsilon_{pen}}{\kappa_0'} - \frac{\mathbf{pen}(m') }{n}$. Define $\kappa_0 = \frac{\theta_{pen}^2 C_{\tau_n}\epsilon_{pen}}{\kappa_0'} $. We have 
\begin{equation*}
\begin{split}
\frac{ L\log( K') }{n} + \frac{y_{m'}^2 C_{\tau_n}\epsilon_{pen}}{\kappa_0'}  - \frac{\mathbf{pen}(m')}{n} &=  \frac{ L\log( K') }{n}+ \kappa_0 \left(\frac{x+x_{m'}}{n} + \sigma_{m'}^2 + \frac{L\log(K')}{n}\right) - \frac{\mathbf{pen}(m')}{n}\\
&\leq \kappa_0 \frac{x}{n} + (1 +  \kappa_0) \left(\frac{x_{m'}}{n} + \sigma_{m'}^2 + \frac{L\log(K')}{n}\right) - \frac{\mathbf{pen}(m')}{n}\\
&\leq \kappa_0 \frac{x}{n}  - (1 - \frac{1 + \kappa_0}{\kappa})\frac{\mathbf{pen}(m')}{n}.
\end{split}
\end{equation*}
Thus, based on equation (\ref{eqPen}), except on a set of probability less than $e^{-x}\Sigma$, simultaneously for any $m'\in \mathcal{M}'$, 
\begin{equation*}
\begin{split}
(1- \epsilon_{pen})\frac{1}{n}\sum_{l =1}^{L}  n_l \mathbf{KL}(s_l, \widehat{f}_{\widehat{k_l}}^{m'} ) + (1 - \frac{1 + \kappa_0}{\kappa})\frac{\mathbf{pen}(m')}{n}
 -\inf_{\Kl}\left(  \frac{1}{n}  \sum_{l=1}^{L}\nu_l(\bar{f}^m_{k_l})  + \inf_{f^m\in\mathcal{F}_m}\sum_{l = 1}^{L}n_l \mathbf{KL}(s_l, f_{k_l}^m) \right) \\
\leq   \frac{\delta}{n} + \frac{\eta  + \eta' }{n}+ \frac{\mathbf{pen}(m)}{n} +  \frac{\kappa_0x}{n} .
\end{split}
\end{equation*}
Since $ \frac{1}{n}  \sum_{l=1}^{L}\nu_l(\bar{f}^m_{k_l}) $ is integrable (with null expectation), we deduce that $M = \sup_{m' \in \mathcal{M}'}\frac{\mathbf{pen}(m')}{n} $ is almost surely finite. By definition of the mapping $\mathbf{pen}$, $\kappa \frac{x_{m'}}{n}\leq M$ for all $m'\in\mathcal{M}'$. Therefore, 
\begin{equation*}
\Sigma \geq \sum_{m'\in\mathcal{M}'} e^{-x_{m'}} \geq \mid \mathcal{M}'\mid e^{-\frac{Mn}{\kappa}},
\end{equation*}
and $\mathcal{M}'$ is almost surely finite. Thus, for all fixed $m\in \mathcal{M}$ fixed, a minimizer $\widehat{m}$ over $\mathcal{M}'$ of 
\begin{equation*}
\mathbf{crit}(m') =  \gamma_n(\widehat{\pi}^{m'}, \widehat{f}^{m'})    + \mathbf{pen}(m') 
\end{equation*}
exists. For this minimizer, with probability greater than $1 - e^{-x}\Sigma$, one has
\begin{equation*}
\begin{split}
(1- \epsilon_{pen})\frac{1}{n}\sum_{l =1}^{L}  n_l \mathbf{KL}(s_l, \widehat{f}_{\widehat{k_l}}^{\widehat{m}} ) + (1 - \frac{1 + \kappa_0}{\kappa})\frac{\mathbf{pen}(\widehat{m})}{n}
-\inf_{\Kl}\left\{  \frac{1}{n}  \sum_{l=1}^{L}\nu_l(\bar{f}^m_{k_l})  + \inf_{f^m\in\mathcal{F}_m}\sum_{l = 1}^{L}n_l \mathbf{KL}(s_l, f_{k_l}^m)\right\} \\
\leq  \frac{\delta}{n}+ \frac{\eta  + \eta' }{n}+  \frac{\mathbf{pen}(m)}{n} + \frac{\kappa_0x}{n}.
\end{split}
\end{equation*}
Using the same integration technique than in the previous theorem, 
\begin{equation*}
\begin{split}
E\left[\frac{1}{n}\sum_{l =1}^{L}  n_l \mathbf{KL}(s_l, \widehat{f}_{\widehat{k_l}}^{\widehat{m}} )+ \frac{(1 - \frac{1 + \kappa_0}{\kappa})}{1- \epsilon_{pen}}\frac{\mathbf{pen}(\widehat{m})}{n} \right]
&\leq \frac{1}{1- \epsilon_{pen}}\left( \inf_{\Kl}\left( \inf_{f^m \in\mathcal{F}_m}\sum_{l = 1}^{L}n_l \mathbf{KL}(s_l, f_{k_l}^m) \right)\right. \\
&\quad+ \left.\frac{\delta}{n}+  \frac{\eta  + \eta' }{n}+ \frac{\mathbf{pen}(m)}{n}  + \frac{\kappa_0\Sigma}{n}\right) , 
\end{split}
\end{equation*}
and since $\delta$ can be artbitrary small, 
\begin{equation*}
\begin{split}
E\left[\frac{1}{n}\sum_{l =1}^{L}  n_l \mathbf{KL}(s_l, \widehat{f}_{\widehat{k_l}}^{\widehat{m}} )+ \frac{(1 - \frac{1 + \kappa_0}{\kappa})}{1- \epsilon_{pen}}\frac{\mathbf{pen}(\widehat{m})}{n} \right]
&\leq \frac{1}{1- \epsilon_{pen}}\left( \inf_{\Kl}\left( \inf_{f^m \in\mathcal{F}_m}\sum_{l = 1}^{L}n_l \mathbf{KL}(s_l, f_{k_l}^m) \right)\right. \\
&\quad+ \left.  \frac{\eta  + \eta' }{n}+ \frac{\mathbf{pen}(m)}{n}  + \frac{\kappa_0\Sigma}{n}\right).
\end{split}
\end{equation*}
This inequality is true for all $m\in\mathcal{M}$. Therefore,
\begin{equation*}
\begin{split}
E\left[\frac{1}{n}\sum_{l =1}^{L}  n_l \mathbf{KL}(s_l, \widehat{f}_{\widehat{k_l}}^{\widehat{m}} )+ \frac{(1 - \frac{1 + \kappa_0}{\kappa})}{1- \epsilon_{pen}}\frac{\mathbf{pen}(\widehat{m})}{n} \right]\\
\leq \frac{1}{1- \epsilon_{pen}}   \inf_{m\in \mathcal{M}} \left(\inf_{\Kl}\left( \inf_{f^m\in\mathcal{F}_m}\sum_{l = 1}^{L}n_l \mathbf{KL}(s_l, f_{k_l}^m) \right)  +\frac{\mathbf{pen}(m)}{n}  \right) 
+ \frac{\kappa_0}{1- \epsilon_{pen}} \frac{\Sigma}{n} + \frac{\eta  + \eta' }{n},
\end{split}
\end{equation*}
and by taking $C_1 = \frac{1}{1-\epsilon_{pen}}$ and $C_2 = \frac{\kappa_0}{1-\epsilon_{pen}}$, we deduce an inequality stronger than the result stated in the theorem because the penalty appears with a positive coefficient on the left-side.
\end{proof}

\subsection{Proofs of Theorems 1 and 2 }

\subsubsection{Bracketing Entropy}
A number of lemmas concerning entropy with bracketing is provided in this paragraph. These results are intended to compute the bracketing entropy of the $K$-product set $\mathcal{F}_K$ with respect to $\mathbf{a}$, based on simpler sets of functions. Let first introduce two other metrics. For $f_k\in\mathcal{F}$, we denote by $\lVert . \lVert_{\infty}$ the $L_\infty$-norm:
\begin{equation*}
\lVert f_k\lVert_\infty = \max_{1 \leq b \leq B} \mid f_k(b)\mid 
\end{equation*}
For $f,\tilde{f}\in \mathcal{F}_K$ and $(k_1,\ldots,k_L) \in \Iintv{1,K}^L$, let also $\mathbf{d}_\infty^2$ be the divergence defined by:
\begin{equation*}
\mathbf{d}_\infty^2(\tilde{f}, f) = \frac{1}{n} \sum_{l=1}^{L}n_l \lVert \log\left( \frac{\tilde{f}_{k_l}}{f_{k_l}}\right)\lVert_{\infty}^2.
\end{equation*}

\begin{lemma}
\label{lemProd}
Let $(\delta_1,\ldots,\delta_K)$ be a family of positive numbers. Then,
\begin{equation*}
H_{[.], \mathbf{d}_\infty^2}(\max_{k = 1,... , K} \delta_k^2, \mathcal{F}_{K} ) \leq  \sum_{k = 1}^{K}      H_{[.], \mathbf{d}_\infty^2}(\delta_k^2, \mathcal{F} ) 
\end{equation*}
\end{lemma}

\begin{proof}
For all $k \in \Iintv{1,K}$, let $\delta_k > 0$ and $[f_k^-, f_k^+]$ a bracket of $\mathbf{d}_\infty^2$-diameter less than $\delta_k^2$ in $ \mathcal{F}$. Then, 
\begin{equation*}
\mathbf{d}_\infty^2( (f_k^-)_{k=1,\ldots,K},  (f_k^+)_{k=1,\ldots,K} ) = \frac{1}{n} \sum_{l=1}^{L}n_l \lVert \log\left( \frac{f_{k_l}^-}{f_{k_l}^+}\right)\lVert_{\infty}^2 \leq \max_{k \in \Iintv{1,K}} \delta_k^2.
\end{equation*}
Therefore, by covering all $\mathcal{F}$ with a number of brackets $N_k$ of width less than $\delta_k^2$, we can cover $\mathcal{F}_K $ with $\prod_{k} N_k$ brackets, which leads to the result.
\end{proof}

\begin{lemma}
\label{lemInfty}
Let $\epsilon > 0$. Then  
\begin{equation*}
H_{[.], \mathbf{d}_\infty^2}(\epsilon^2, \mathcal{F}) \leq    H_{[.], \lVert.\lVert_\infty}(\epsilon, [-\tau_n, 0]^B). 
\end{equation*}
\end{lemma}

\begin{proof}
By definition, $\mathcal{F} = \mathbb{S}_{B-1} \cap [e^{-\tau_n}, 1 ]^B$. Let $\mathcal{B}_k(\epsilon)$ be a set of brackets of $\lVert . \lVert_{\infty}$-width less than $\epsilon$ covering $ [-\tau_n, 0]^B$. Let $f_k\in\mathcal{F}$. Then, $\log(f_k)\in[-\tau_n, 0]^B$ and there exists a bracket $[u_k^-, u_k^+] \in \mathcal{B}_k(\epsilon)$ such that $u_k^-\leq \log(f_k) \leq u_k^+$ with $\lVert u_k^- - u_k^+\lVert_\infty \leq \epsilon $. We can rewrite $u_k^-$ and $u_k^+$ as $u_k^- = \log(v_k^-)$, $u_k^+ = \log(v_k^+)$ respectively. This leads to a bracket $[v_k^-, v_k^+]$ of width less than $\epsilon^2$ with respect to $\mathbf{d}_\infty^2$. Therefore, an $\epsilon-\lVert.\lVert_\infty$-covering of $[-\tau_n, 0]^B$ induces an $\epsilon^2-\mathbf{d}_\infty^2$-covering of $\mathcal{F}$ so we can conclude.
\end{proof} 

The following result is inspired by Lemma 2 from \cite{GenWas2000}.

\begin{lemma}
\label{lemValue}
Let $\epsilon >0$. Then
\begin{equation*}
H_{[.], \lVert.\lVert_\infty}(\epsilon, [-\tau_n, 0]^B) \leq B\log(2) + B  \left(\log\left( \frac{\tau_n}{\epsilon}\right) \right)_+
\end{equation*}
\end{lemma}

\begin{proof}
Divide the cube $[-\tau_n, 0]^B$ of $\mathbb{R}^B$ into a number of $N$ disjoint cubes with sides parallels to the axes and of length $\epsilon$. For one cube, let $x_1$ the closest vertex from 0, and $y_1$ the furthest vertex from 0. We have $\max_{b \in\Iintv{1,B}}\mid x_1(b) - y_1(b)\mid  \leq  \epsilon $. Thus, the family of vertices $\{(x_1, y_1), ..., (x_N, y_N)\}$ forms an $\epsilon-\lVert.\lVert_\infty$ bracketing of $[-\tau_n, 0]^B$. Clearly, we have 
\begin{equation*}
N \leq \left( 1 + \frac{\tau_n}{\epsilon} \right)^B \leq \max \left( 2^B, \left(\frac{2\tau_n}{\epsilon}\right)^B   \right) \leq 2^B\max\left(1, \frac{\tau_n}{\epsilon}\right)^B.
\end{equation*}
\end{proof}

\begin{proposition}[\textbf{Bracketing entropy of }$\mathbf{\mathcal{F}_K}$]
\label{brackFm}
For all $K\in\mathcal{M}$ and all $\delta \in (0, 1]$, we have
\begin{equation*}
H_{[.], \mathbf{a}}(\delta, \mathcal{F}_K) \leq K B \left( \log(2\tau_n) + \log\left(   \frac{1}{\delta}  \right) \right).
\end{equation*}
\end{proposition}

\begin{proof}[Proof]
Let $\delta \in (0, 1]$. By definition, $ \mathbf{a} \leq \mathbf{d}_\infty^2$. Therefore, $H_{[.], \mathbf{a}}(\delta^2, \mathcal{F}_K) \leq H_{[.], \mathbf{d}_\infty^2}(\delta^2, \mathcal{F}_K)$. Recalling that $\mathcal{F} = \mathbb{S}_{B-1} \cap [e^{-\tau_n}, 1 ]^B$ and given Lemmas \ref{lemProd}, \ref{lemInfty} and \ref{lemValue}, we have
\begin{equation*}
\begin{split}
H_{[.], \mathbf{d}_\infty^2}(\delta^2, \mathcal{F}_K) &\leq K   H_{[.], \mathbf{d}_\infty^2} (\delta^2, \mathbb{S}_{B-1} \cap [e^{-\tau_n}, 1 ]^B) \\
&\leq  K  H_{[.], \lVert.\lVert_\infty}(\delta, [-\tau_n, 0]^B) \\
&\leq K \left( B\log(2) + B\left(\log\left( \frac{\tau_n}{\delta}\right)\right)_+\right)
\end{split}
\end{equation*} 
and $H_{[.], \mathbf{a}}(\delta, \mathcal{F}_K)\leq KB \left( \log(2) + \left(\log\left( \frac{\tau_n}{\delta}\right)\right)_+\right) \leq KB \left( \log(2\tau_n) + \log\left( \frac{1}{\delta}\right)\right)$ because $\delta \in (0, 1]$.
\end{proof}

\subsubsection{Bracketing and model dimension}
Theorems \ref{thmSingleModelParticular} and \ref{thmModelSelectionParticular} are obtained from Theorems \ref{thmSingleModel} and \ref{thmModelSelection} which address a penalty function related to geometrical properties of the models, namely bracketing entropy with respect to some distance \textbf{a}. Recall that the function $\sigma \mapsto \int_{0}^{\sigma} H_{[.], \mathbf{a}}(\delta, \mathcal{S}_m)^{1/2}d\delta$ always satisfies Assumption \ref{hypPhi}. The quantity $\mathfrak{D}_m$ is defined as $n\sigma_m^2$ where $\sigma_m^2$ is the unique root of $\phi_m(\sigma)/\sigma = n^{1/2}\sigma$. A good choice of $\phi_m$ is one which leads to a small upper bound of $\mathfrak{D}_m$. Although $\sigma_m$ is not very explicit, it can be related to an entropic dimension of the model.\\

Define the bracketing dimension $D_m$ of a compact set as the smallest real number $D$ such that there exists a constant $C$ such that
 
\begin{equation*}
\forall \delta > 0 , H_{[.], \mathbf{a}}(\delta, \mathcal{F}_m) \leq D(C + \log\left( \frac{1}{\delta}\right)).
\end{equation*}
In a parametric setting, the bracketing dimension is equivalent to the number of parameters to be estimated within a model. The following result from \cite{CohLep2011} states that under some assumption on the bracketing entropy, $\mathfrak{D}_m$ is proportional to the entropic dimension $D_m$.

\begin{proposition}
\label{dimELPC}
Assume for any $\delta \in (0, 1]$, there exist $D_m>0$ and $C_m\geq 0$ such that
\begin{equation*}
H_{[.], \mathbf{a}}(\delta, \mathcal{F}_m) \leq D_m \left(C_m + \log\left(   \frac{1}{\delta}  \right) \right).
\end{equation*}
Then, the function 
\begin{equation*}
\phi_m(\sigma) = \sigma D_m^{1/2} \left(   C_m^{1/2} + \pi^{1/2}  + \log\left(    \frac{1}{\sigma \land e^{-1/2} } \right)^{1/2} \right) 
\end{equation*}
satisfies the properties required in Assumption \ref{hypPhi} and $\mathfrak{D}_m$ satisfies
\begin{equation*}
\mathfrak{D}_m \leq    \left(    2\left(   C_m^{1/2}  + \pi^{1/2}  \right)^2 +1 +\log \left(   \frac{n}{e(C_m^{1/2} + \pi^{1/2})^2D_m}  \right)_+   \right) D_m.
\end{equation*}
\end{proposition}

\subsubsection{Proof of Theorem 1}
Proposition \ref{brackFm} directly indicates that if we choose the constants $D_m = K B$ and $C_m = \log(2\tau_n)$ and apply Proposition \ref{dimELPC}, then the given $\phi_m$ satisfies Assumption \ref{hypPhi}  in our setting.
Therefore, we can apply Theorem \ref{thmSingleModel} and get the following oracle inequality:
\begin{equation*}
\begin{split}
E\left[\sum_{l = 1}^{L} \frac{n_l}{n} \mathbf{KL}(s_l, \widehat{f_{\widehat{k_l}}}) \right]   &\leq  C_1  \left( \inf_{\substack{f \in \mathcal{F}_K\\(k_l)_l \in \Iintv{1,K}^L}} \left(\sum_{l=1}^{L} \frac{n_l}{n} \mathbf{KL}(s_l, f_{k_l})\right)        + (2+\kappa_0)\frac{  L\log K}{n} + \kappa_0 \frac{\mathfrak{D}_{m}}{n}  \right)\\
&\quad+ \frac{C_2}{n} + \frac{ \eta}{n}.
\end{split}
\end{equation*}
By definition, $\mu_n =  2 \left(   \log(2\tau_n)^{1/2}  + \pi^{1/2}  \right)^2 + 1 + \log(n)$. According to Proposition \ref{dimELPC}, we also have
\begin{equation*}
\begin{split}
E\left[\sum_{l = 1}^{L} \frac{n_l}{n} \mathbf{KL}(s_l, \widehat{f_{\widehat{k_l}}}) \right]   &\leq  C_1  \left(   \inf_{\substack{f \in \mathcal{F}_K\\(k_l)_l \in \Iintv{1,K}^L}} \left(\sum_{l=1}^{L}  \frac{n_l}{n} \mathbf{KL}(s_l, f_{k_l})\right)   +  (2+\kappa_0)\frac{  L\log K}{n} +  \frac{\kappa_0\mu_n D_K}{n}  \right)\\
&\quad+ \frac{C_2}{n} + \frac{ \eta}{n}.
\end{split}
\end{equation*}
It remains to choose $\lambda_0 =2 +  \kappa_0$ to conclude the proof.

\subsubsection{Proof of Theorem 2}
We need to find the weights $x_m$ satisfying Assumption  \ref{hypKraft} in order to apply Theorem \ref{thmModelSelection} in our framework. It is easy to show that $x_m \geq K \log(2)$ is a sufficient condition on the weights. Indeed, define $\delta = 1/2$. Then, $e^{-x_m} \leq \delta^{K}$. With a collection $\mathcal{M} = \mathbb{N} \setminus\{0\}$, we have
\begin{equation*}
\begin{split}
\sum_{m\in\mathcal{M}}e^{-x_m} &\leq \sum_{K\geq 1}\delta^{K}  =  \frac{\delta}{1 - \delta} = 1.
\end{split}
\end{equation*}
We thus take $x_m  = K\log(2)$. Then we have $\sum_{m\in\mathcal{M}}e^{-x_m} \leq 1$. We can apply Theorem \ref{thmModelSelection} and use Proposition \ref{dimELPC} to state that for any chosen $\lambda'_0 > \kappa_0+1$, if
\begin{equation*}
\mathbf{pen}(K) = \lambda'_0(\mu_n D_{K} + L \log(K) +  x_m), 
\end{equation*}
then the following inequality is satisfied:
\begin{equation*}
\begin{split}
E\left[\sum_{l =1}^{L}  \frac{n_l}{n} \mathbf{KL}(s_l, \widehat{f_{\widehat{k_l}}} ) \right]
&\leq C_1   \inf_{K\in \mathcal{M}} \left( 
\inf_{\substack{f \in \mathcal{F}_K\\(k_l)_l \in \Iintv{1,K}^L}} \sum_{l = 1}^{L}\frac{n_l}{n} \mathbf{KL}(s_l, f_{k_l})  + \frac{\mathbf{pen}(K)}{n}   \right)\\
&\quad+  \frac{C_2}{n} + \frac{\eta  + \eta' }{n}.
\end{split}
\end{equation*}
This concludes the proof.

\subsubsection{Varying number of categories}
\label{varyingCat}
In this paragraph, the number $B$ of different multinomial parameters varies besides the number of clusters $K$. We consider an ordered set of $B_{\text{max}} \geq B$ categories where the first $B$ categories can have different proportions that sum up to less than $1$, whereas the others are uniformly distributed over the remaining probability. We still assume the lower bound of Assumption \ref{GhypSurModele} on density distributions. Thus, in this case, the set of density functions to be considered is defined by:
\begin{equation*}
\begin{split}
\mathcal{F}_{(K,B)} &= \mathcal{F}^K \quad\text{with}\\
\mathcal{F}&=\{(f(b))_{1\leq b \leq B_{\text{max}}} \mid \sum_{b = 1}^{B_{\text{max}}} f(b) = 1 , e^{-\tau_n}\leq f(b) \leq 1,  f(B+1) = ... = f(B_\text{max}) \}.
\end{split}
\end{equation*}
In this framework, the following bound on the bracketing entropy can be stated:

\begin{proposition}[\textbf{Bracketing entropy of }$\mathbf{\mathcal{F}_{(K,B)}}$]
\label{brackFKB}
For all $\delta \in (0, 1]$, we have
\begin{equation*}
H_{[.], \mathbf{a}}(\delta, \mathcal{F}_{(K,B)}) \leq K (B+1) \left( \log(2\tau_n) + \log\left(   \frac{1}{\delta}  \right) \right).
\end{equation*}
\end{proposition}

\begin{proof}[Proof]
Since $ \mathbf{a} \leq \mathbf{d}_\infty^2$, we still have $H_{[.], \mathbf{a}}(\delta^2, \mathcal{F}_{(K,B)}) \leq H_{[.], \mathbf{d}_\infty^2}(\delta^2, \mathcal{F}_{(K,B)})$. Following Lemma \ref{lemProd}, we also have
\begin{equation*}
\begin{split}
H_{[.], \mathbf{d}_\infty^2}(\delta^2, \mathcal{F}_{(K,B)}) &\leq K   H_{[.], \mathbf{d}_\infty^2} (\delta^2,\mathcal{F}) 
\end{split}
\end{equation*} 
Remarking that $B+1$ parameters determine elements of $\mathcal{F}$ ($B$ first multinomial parameters plus one that is uniformly distributed over the remaining categories), Lemmas \ref{lemInfty} and \ref{lemValue} lead to
\begin{equation*}
\begin{split}
H_{[.], \mathbf{d}_\infty^2}(\delta^2, \mathcal{F}_{(K,B)}) &\leq K   H_{[.], \mathbf{d}_\infty^2} (\delta^2,\mathcal{F}) \\
&\leq  K  H_{[.], \lVert.\lVert_\infty}(\delta, [-\tau_n, 0]^{B+1}) \\
&\leq K \left( (B+1)\log(2) + (B+1)\left(\log\left( \frac{\tau_n}{\delta}\right)\right)_+\right),
\end{split}
\end{equation*} 
and $H_{[.], \mathbf{a}}(\delta, \mathcal{F}_K)\leq K(B+1) \left( \log(2) + \left(\log\left( \frac{\tau_n}{\delta}\right)\right)_+\right) \leq K(B+1) \left( \log(2\tau_n) + \log\left( \frac{1}{\delta}\right)\right)$.
\end{proof}

Theorem \ref{thmModelSelection} can thus be applied to this case, provided that the Kraft type assumption \ref{hypKraft} is satisfied. A model is defined as $m = (K,B)$ with $(K,B)\in \mathcal{M}$, $\mathcal{M}$ being the following collection of models: 
\begin{equation*}
\mathcal{M} = \{(1,1)\} \cup \mathbb{N}\setminus \{0\} \times \mathbb{N}\setminus \{0,1\}.
\end{equation*} 
As a matter of fact, 

\begin{proposition}
If $x_m \geq KB\log(2)$, then Kraft Assumption  \ref{hypKraft} is satisfied and $\sum_{m\in\mathcal{M}}e^{-x_m} \leq 1$.
\end{proposition}

\begin{proof}
Define $\delta = 1/2$. Then, $e^{-x_m} \leq \delta^{KB}$ and
\begin{equation*}
\begin{split}
\sum_{m\in\mathcal{M}}e^{-x_m} &\leq \delta + \sum_{K\geq 1, B\geq 2}\delta^{KB}  = \delta + \sum_{B\geq 2} \frac{\delta^B}{1-\delta^B}
\leq  \delta + \sum_{B\geq 2} \frac{\delta^B}{1-\delta}
= \delta + \frac{\delta^2}{1 - \delta} = 1.
\end{split}
\end{equation*}
\end{proof}

The shape of the penalty thus obtained is given by $\mathbf{pen}(K,B) = \lambda'_0(\mu_n K(B+1) + L \log(K) + KB\log(2))$ with $\lambda'_0 >\kappa_0+1$.

\end{document}